\documentclass[11pt]{amsart}
\usepackage{amsmath, amsthm, url,stmaryrd}
\usepackage[all]{xy}

\numberwithin{equation}{section}
\newtheorem{theorem}[equation]{Theorem}
\newtheorem{lemma}[equation]{Lemma}
\newtheorem{cor}[equation]{Corollary}
\theoremstyle{definition}
\newtheorem{defn}[equation]{Definition}
\newtheorem{remark}[equation]{Remark}
\newtheorem{example}[equation]{Example}

\newcommand{\calB}{\mathcal{B}}
\newcommand{\calF}{\mathcal{F}}
\newcommand{\calG}{\mathcal{G}}
\newcommand{\calH}{\mathcal{H}}
\newcommand{\calM}{\mathcal{M}}
\newcommand{\calO}{\mathcal{O}}
\newcommand{\calP}{\mathcal{P}}
\newcommand{\frakm}{\mathfrak{m}}
\newcommand{\frako}{\mathfrak{o}}
\newcommand{\frakp}{\mathfrak{p}}
\newcommand{\frakq}{\mathfrak{q}}
\newcommand{\frakU}{\mathfrak{U}}
\newcommand{\frakV}{\mathfrak{V}}
\newcommand{\FF}{\mathbb{F}}
\newcommand{\QQ}{\mathbb{Q}}
\newcommand{\RR}{\mathbb{R}}
\newcommand{\ZZ}{\mathbb{Z}}

\DeclareMathOperator{\Aut}{Aut}
\DeclareMathOperator{\catPrespec}{\mathbf{Prespec}}
\DeclareMathOperator{\catSpec}{\mathbf{Spec}}
\DeclareMathOperator{\coker}{coker}
\DeclareMathOperator{\Comm}{Comm}
\DeclareMathOperator{\Cont}{Cont}
\DeclareMathOperator{\DLat}{\mathbf{DLat}}
\DeclareMathOperator{\Frac}{Frac}
\DeclareMathOperator{\Gr}{Gr}
\DeclareMathOperator{\Maxspec}{Maxspec}
\DeclareMathOperator{\op}{op}
\DeclareMathOperator{\Spa}{Spa}
\DeclareMathOperator{\Spra}{Spra}
\DeclareMathOperator{\Sprb}{Sprb}
\DeclareMathOperator{\Sprv}{Sprv}
\DeclareMathOperator{\Spv}{Spv}
\DeclareMathOperator{\Spec}{Spec}

\begin{document}

\title{Reified valuations and adic spectra}
\author{Kiran S. Kedlaya}
\thanks{Supported by NSF (grant DMS-1101343) and UC San Diego
(Warschawski Professorship), and by MSRI during fall 2014
(via NSF grant DMS-0932078). Thanks to Antoine Ducros, Roland Huber, Thomas Scanlon, and Michael Temkin for helpful discussions.}
\date{August 4, 2015}

\begin{abstract}
We revisit Huber's theory of continuous valuations, which give rise to the adic spectra used in his theory of adic spaces. We instead consider valuations which have been \emph{reified}, i.e., whose value groups have been forced to contain the real numbers. This yields \emph{reified adic spectra} which provide a framework for an analogue of Huber's theory compatible with Berkovich's construction of nonarchimedean analytic spaces. As an example, we extend the theory of perfectoid spaces to this setting.
\end{abstract}

\maketitle

There are several frameworks for analytic geometry over nonarchimedean fields, which can be classified into roughly three types:
\begin{itemize}
\item
\emph{rigid analytic geometry} (Tate), which can also be obtained via \emph{formal geometry with admissible blowups} (Raynaud);
\item
\emph{nonarchimedean analytic geometry} (Berkovich), which can also be obtained via \emph{tropical geometry} (Payne et al.);
\item
\emph{adic geometry} (Huber), which can also be obtained via formal geometry (Abbes, Fujiwara-Kato).
\end{itemize}
For a comparative discussion (primarily between the first two viewpoints), see \cite{conrad}. Here, we limit ourselves to an instructive analogy: the three frameworks give results analogous to those of the following three constructions.
\begin{itemize}
\item
Consider the rational numbers with the Grothendieck topology of finite unions of closed intervals
with rational endpoints. Let $T_1$ be the resulting topos.
\item
Consider the real numbers. The Grothendieck topology of finite unions of closed intervals with rational endpoints
recovers $T_1$. The natural topology defines a new topos $T_2$.
The Grothendieck topology of finite unions of all closed intervals 
defines a new topos $T_3$.
\item
Consider the real numbers plus some additional points $r \pm \epsilon$ for each rational number $r$. The natural
topology recovers the topos $T_1$.
\end{itemize}
In this paper, we introduce a construction playing the role of the
real numbers plus points $r \pm \epsilon$ for each \emph{real} number $r$, whose natural topology recovers the 
topos $T_3$. This makes it possible to overcome a mismatch between the theories of Berkovich and Huber:
while Huber's theory is based on the classical theory of Krull valuations, Berkovich's theory is based on 
real-valued seminorms. The link comes via the fact that any rank 1 Krull valuation can be interpreted as a
real valuation; however, one can rescale a real valuation
without changing the equivalence class of the underlying Krull valuation. To correct this, we consider \emph{reified valuations}, for which we fix the comparisons between real numbers and elements
of the value group. These also appear in upcoming work of Ducros and Thuillier on the relationship between monomial valuations and skeleta of Berkovich spaces
\cite{ducros-thuillier}.

Using reified valuations, one can simulate much of the analysis of continuous valuations
from \cite{huber1} and the comparison of rigid and adic spaces by Huber \cite[\S 4]{huber1} and van der Put and Schneider \cite{vanderput-schneider}. In fact, in some ways the reified version of the analysis is somewhat simpler. For example,
when working with Banach algebras over an ultrametric field, the case of a trivially valued field can be handled
more uniformly using reified valuations; this is consistent with the corresponding
uniformity in Berkovich's theory. Roughly speaking, reification provides an alternative to the use of topologically nilpotent units, such as in Tate's fundamental theorem on the acyclicity of the structure sheaf.

We also describe the structure presheaf on a reified adic spectrum and carry out some of the local preliminary work to a theory of \emph{reified adic spaces}.
As in Huber's construction of adic spaces, the construction of reified adic spaces involves topological rings plus some auxiliary data. In Huber's construction, the auxiliary datum associated to a topological ring is a certain \emph{subring of integral elements}; the analogous datum in our setting is defined in terms of the graded ring associated to a nonarchimedean Banach space. (The graded ring first appeared prominently in  work of Temkin extending some key properties of rigid analytic spaces to Berkovich spaces \cite{temkin-local2}, so its appearance here is perhaps not surprising.)

For the reified adic space associated to a single ring, we establish a Tate-style acyclicity theorem for the structure sheaf and a Kiehl-style glueing theorem for vector bundles (and for coherent sheaves under a suitable noetherian hypothesis), following \cite[\S 2]{kedlaya-liu1}.
One important point is that in the context of Berkovich spaces, these results apply to coverings for the full G-topology, as shown in \cite{berkovich1}; by contrast, by passing from Berkovich spaces to adic spaces, one only obtains acyclicity and glueing with respects to coverings for the strictly analytic G-topology. For an explicit example, pick $0 < \rho_2 \leq \rho_1$ and consider the disc $\left| T \right| \leq \rho_1$: in the full G-topology this disc admits an admissible covering by the disc $\left| T \right| \leq \rho_2$ and the annulus $\rho_2 \leq \left| T \right| \leq \rho_1$; in the strictly analytic G-topology, this only occurs if $\rho_1$ belongs to the divisible closure of the norm group of the base field.

As an illustration, we describe the spaces associated to perfectoid algebras; this amounts to a fairly faithful translation of certain sections of \cite{kedlaya-liu1}.
In fact, this paper was borne out of the author's frustration with the status quo during the writing of \cite{kedlaya-liu1}: while in many respects it is natural to study perfectoid algebras via their Gel'fand transforms, these cannot be easily glued without promoting them to something like adic spaces, and at the time no such construction was available in the literature.

To conclude this introduction, we mention two related constructions. Instead of fixing comparisons between elements of the value group and arbitrary positive real numbers, one may only fix these comparisons for real numbers in some multiplicative subgroup $H$;
this yields the concept of \emph{$H$-reified valuations}, which interpolates between ordinary valuations and our reified valuations. One can easily modify our arguments to produce statements about such valuations, but we have not done so (despite such valuations making an appearance in \cite{ducros-thuillier}). In a different direction,
Foster and Ranganathan \cite{foster-ranganathan} have described an analogue of tropicalization in which the real numbers are replaced by the value group of a valuation of possibly higher rank; reified valuations provide a natural context for comparing this construction to ordinary tropicalization.

\section{Spectral and prespectral spaces}
\label{sec:prime filters}

We first generalize Hochster's formalism of spectral spaces \cite{hochster} to G-topological spaces. This links the spaces considered by Huber with other types of analytic spaces; see for example
\cite[\S 4]{huber1}.

\begin{defn}
A \emph{bounded distributive lattice} is a partially ordered set $D$ satisfying the following conditions.
\begin{enumerate}
\item[(a)]
The set $D$ has a least element $0$ and a greatest element $1$.
\item[(b)]
For any $x,y \in D$, the set of $z \in D$ for which $z \leq x, z \leq y$ has a greatest element
$x \wedge y$ (the \emph{meet} of $x$ and $y$).
\item[(c)]
For any $x,y \in D$, the set of $z \in D$ for which $z \geq x, z \geq y$ has a unique least element
$x \vee y$ (the \emph{join} of $x$ and $y$).
\item[(d)]
The meet and join operations are distributive over each other.
\end{enumerate}
Let $\DLat$ denote the category whose objects are bounded distributive lattices and whose
morphisms are maps of sets preserving $\leq, 0, 1, \wedge, \vee$.
\end{defn}

\begin{defn} \label{D:filter}
A \emph{filter} on $D \in \DLat$ is a subset $\calF$ of $D$ satisfying the following conditions.
\begin{enumerate}
\item[(a)]
We have $1 \in \calF$ and $0 \notin \calF$.
\item[(b)]
For any $S_1, S_2 \in \calF$, we have $S_1 \wedge S_2 \in \calF$.
\item[(c)]
For any $S \in \calF$, any $T \in D$ with $T \geq S$ is also in $\calF$.
\end{enumerate}
A filter $\calF$ on $D$ is \emph{prime} if it satisfies the following additional condition.
\begin{enumerate}
\item[(d)]
For any $S_1, \dots, S_n \in D$ with $S_1 \vee \cdots \vee S_n \in \calF$,
there exists $i \in \{1,\dots,n\}$ for which $S_i \in \calF$.
\end{enumerate}
Let $\Spec(D)$ be the set of prime filters on $D$ equipped with the topology generated by
the sets $\tilde{S} := \{\calF \in \Spec(D): S \in \calF\}$ for $S \in D$.
\end{defn}

We now recall some relevant properties of G-topological spaces.
\begin{defn}
Let $X$ be a G-topological space in the sense of \cite[Definition~9.1.1/1]{BGR},
i.e., a Grothendieck topology whose underlying category is a family of subsets of $X$ closed under pairwise intersections. (In practice, it is harmless to assume that the family is closed under finite intersections, as this only adds the condition that $X$ itself is an open subset.)
We say $X$ is \emph{$T_0$} if for any $x \neq y \in X$,
there exists an open subset of $X$ containing exactly one of $x,y$.

A nonempty closed subspace $Z$ of $X$ is \emph{irreducible} if
 for any two open subsets $U,V$ such that $U \cap Z$ and $V \cap Z$
are nonempty, $U \cap V \cap Z$ is also nonempty; it is enough to check this for $U,V$ running through a basis.

The closure $Z$ of a point $x \in X$ is irreducible: an open set meets $Z$ if and only
if it contains $x$. We say $X$ is \emph{sober} if conversely any irreducible closed subset of $X$
is the closure of a unique point of $X$; any sober space is $T_0$.
\end{defn}

We now introduce the concepts of spectral and prespectral spaces.
\begin{defn} \label{D:prespectral}
Let $X$ be a G-topological space.
A \emph{basis} of $X$ is a family $\calB$ of open sets such that every (admissible) open subset of $X$
admits an admissible covering by elements of $\calB$. 

We say that $X$ is \emph{quasiseparated} (or \emph{semispectral} in the language of \cite{hochster}) if the intersection of any two quasicompact open subsets of $X$
is again quasicompact. 
We write \emph{qcqs} as shorthand for \emph{quasicompact and quasiseparated}.

We say that $X$ is \emph{prespectral} if 
$X$ is qcqs, any finite union of quasicompact open sets is open, and the quasicompact open sets form a basis. In particular, the quasicompact open subsets of $X$ form a bounded distributive lattice with $0 = \emptyset$, $1 = X$, $\wedge = \cap$, $\vee = \cup$.

We say that $X$ is \emph{spectral} if the G-topology on $X$ is an ordinary topology (that is, any union of open sets is open) and $X$ is both
prespectral and sober. Spectral spaces are called \emph{coherent spaces} in some sources, such as \cite{johnstone}.

A map $f: X \to Y$ between prespectral spaces is \emph{spectral} if the preimage of any quasicompact open subset is a quasicompact open. If $X$ is a topological space, this forces $f$
to be continuous.
Let $\catPrespec$ (resp.\ $\catSpec$) be the category of prespectral (resp.\ spectral) spaces and spectral morphisms.
\end{defn}


The key property of spectral spaces is the following result of topos theory.
\begin{theorem}[Stone duality] \label{T:Stone duality}
There is an equivalence of categories 
\[
\DLat^{\op} \sim \catSpec
\]
which in one direction takes $D \in \DLat^{\op}$ to $\Spec(D)$ and in the other takes
$X \in \catSpec$ to the lattice $D(X)$ of quasicompact open subsets of $X$.
\end{theorem}
\begin{proof}
The functor $\DLat^{\op} \to \catSpec$ acts on morphisms via pullback:
for $f: D_1 \to D_2$ a morphism in $\DLat$ and $\calF \in \Spec(D_2)$,
the set $\{S \in D_1: f(S) \in D_2\}$ is a prime filter on $D_1$.
For more, see \cite[Corollary~II.3.4]{johnstone}.
\end{proof}

\begin{cor} \label{C:prespec to spec}
The forgetful functor $\catSpec \to \catPrespec$ admits a left adjoint 
taking $X \in \catPrespec$ to $\Spec(D(X))$ for $D(X)$ the lattice
of quasicompact open subsets of $X$.
\end{cor}
\begin{proof}
For $X \in \catPrespec$, the adjunction map $X \to \Spec(D(X))$
takes $x \in X$ to the prime filter $\{S \in D(X): x \in S\}$. 
(Note that this map is spectral, but not necessarily continuous if $X$ is not an ordinary topological space.)
On the other side, for $Y \in \catSpec$, Theorem~\ref{T:Stone duality} provides a
natural isomorphism $\Spec(D(Y)) \cong Y$ for which the composition
$Y\to \Spec(D(Y)) \to Y$ is the identity map.
\end{proof}

\begin{defn}
For $X$ a topological space, the \emph{patch topology} (or \emph{constructible topology}) on $X$
is the new topology generated by the open sets and complements of quasicompact open sets
of the original topology. If $X$ is a spectral space, we sometimes call its original topology the \emph{spectral topology} to distinguish it from the patch topology.
\end{defn}

The key property of the patch topology is the following \cite[Theorem~1]{hochster}.
\begin{theorem} \label{T:spectral quasicompact intersections}
Any spectral space is compact under the patch topology.
\end{theorem}
\begin{proof}
It is clear that the patch topology is Hausdorff. To check quasicompactness, it suffices to check that
any family of closed and quasicompact open sets for the spectral topology
which is maximal for the finite intersection property has nonempty intersection.
But the intersection of the closed members of such a family is irreducible (by maximality) 
and so has a generic point, which belongs to the full
intersection.
\end{proof}
\begin{cor}
An open subset of a spectral space is quasicompact for the spectral topology if and only if it is closed-open for the patch topology.
\end{cor}

\begin{cor} \label{C:continuous is spectral}
A continuous map of spectral spaces is spectral if and only if it is continuous for the patch topologies.
\end{cor}

\begin{cor} \label{C:quasicompact is spectral}
A topological space which is $T_0$ and prespectral is spectral if and only if its patch topology
is quasicompact.
\end{cor}
\begin{proof}
For any irreducible closed subspace $Z$, any point in the 
intersection of the quasicompact open subsets of $Z$ is a generic point.
\end{proof}

\begin{remark} \label{R:adjunction}
We collect some additional observations about the adjunction map $X \to \Spec(D(X))$ associated to
$X \in \catPrespec$ via Corollary~\ref{C:prespec to spec}.
\begin{enumerate}
\item[(a)]
This map is the unique (up to unique isomorphism) morphism $f: X \to Y$ in $\catPrespec$ with
$Y \in \catSpec$ such that the image of $f$ is dense under the patch topology
and $X$ admits a basis consisting of inverse images of quasicompact open subsets of $Y$.
(Namely, Corollary~\ref{C:prespec to spec} produces a morphism $\Spec(D(X)) \to Y$ in $\catSpec$.
For the patch topologies, we have a continuous map between compact spaces which is injective with dense image,
hence a homeomorphism.)
\item[(b)]
This map is the unique (up to unique isomorphism) morphism $f: X \to Y$ in $\catPrespec$
with $Y \in \catSpec$ defining an isomorphism of topoi. (This reduces easily to (a).)
\item[(c)]
This map is injective if and only if $X$ is $T_0$.
Consequently, Corollary~\ref{C:prespec to spec} is a refinement of \cite[Theorem~8]{hochster},
which asserts that an ordinary topological space is \emph{spectralifiable} (prespectral and $T_0$) if and only if it can be spectrally embedded into some spectral space.
\end{enumerate}
\end{remark}

\begin{lemma} \label{L:extract spectral}
Let $(X,T)$ be a compact topological space. Let $F \subseteq T$ be a family of closed-open subsets of $X$.
Suppose that the topology $T'$ on $X$ generated by $F$ is $T_0$.
Then $(X,T')$ is a spectral space in which the elements of $F$ are quasicompact open.
\end{lemma}
\begin{proof}
See \cite[Proposition~7]{hochster}.
\end{proof}
\begin{cor} \label{C:closed in spectral}
Any subspace of a spectral space which is closed under the patch topology is a spectral space.
\end{cor}

\begin{remark}
Although we will not use this fact, it is worth noting that a topological space is spectral if and only if it is isomorphic to the prime
spectrum of a commutative ring \cite[Theorem~6]{hochster}.
\end{remark}

\section{Spaces of valuations}
\label{sec:valuations}

Throughout \S\ref{sec:valuations}, fix a (commutative unital) ring $A$.
We recall the construction and basic properties of the space of valuations on $A$, following \cite{huber1}.
Given our goals, it is natural to write valuations and semivaluations multiplicatively rather than additively;
this is inconsistent with classical literature
on valuation theory, but it is consistent with Huber's papers.

\begin{defn}
By a \emph{value group}, we will mean a totally ordered abelian group written multiplicatively (so that $1$
is its identity element). For $\Gamma$ a value group,
let $\Gamma_0$ denote the pointed commutative monoid $\Gamma \cup \{0\}$
ordered so that $0 < \gamma$ for all $\gamma \in \Gamma$.
\end{defn}

\begin{defn}
A \emph{semivaluation} on the ring $A$ is a function $v: A \to \Gamma_0$
for some value group $\Gamma$ satisfying the following conditions.
\begin{enumerate}
\item[(a)]
We have $v(0) = 0$ and $v(1) = 1$.
\item[(b)]
For all $x,y \in A$, we have $v(x+y) \leq \max\{v(x), v(y)\}$.
\item[(c)]
For all $x,y \in A$, we have $v(xy) = v(x) v(y)$.
\end{enumerate}
For $v$ a semivaluation, the \emph{kernel} of $v$ is the prime ideal $v^{-1}(0)$;
note that $v$ induces a Krull valuation on $\Frac(A/v^{-1}(0))$.

For $v$ a semivaluation, let $\Gamma_{v,0}$ be the image of $v$ and put $\Gamma_v := \Gamma_{v,0} \setminus \{0\}$.
Two semivaluations $v_1, v_2$ on $A$ are \emph{equivalent} if 
there exists an isomorphism $i: \Gamma_{v_1} \cong \Gamma_{v_2}$ of value groups 
(which we also view as an isomorphism $i: \Gamma_{v_1,0} \cong \Gamma_{v_2,0}$)
such that $i \circ v_1 = v_2$. The equivalence classes of semivaluations on $A$ then correspond to pairs $(\frakp, \frako)$ in which
$\frakp$ is a prime ideal of $A$ and $\frako$ is a valuation ring of $\Frac(A/\frakp)$.
\end{defn}

\begin{defn}
The \emph{valuative spectrum} of $A$ is the set
$\Spv(A)$ of equivalence classes of semivaluations on $A$,
equipped with the topology generated by sets of the form
\begin{equation} \label{eq:valuative spectrum basis}
\{v \in \Spv(A): v(a) \leq v(b) \neq 0\} \qquad (a,b \in A).
\end{equation}
Let $\calB$ be the Boolean algebra generated by sets of the form \eqref{eq:valuative spectrum basis};
note that $\calB$ is also generated by the sets
\begin{equation} \label{eq:valuative spectrum patch basis}
\{v \in \Spv(A): v(a) \leq v(b)\} \qquad (a,b \in A).
\end{equation}
\end{defn}

\begin{lemma} \label{L:valuative is spectral}
The space $\Spv(A)$ is spectral
and the elements of $\calB$ are compact for the patch topology.
In particular, any finite intersection of subspaces as in \eqref{eq:valuative spectrum basis}
is quasicompact and open.
\end{lemma}
\begin{proof}
We follow the proof of \cite[Proposition~2.2]{huber1}.
We first observe that distinct elements of $\Spv(A)$ can be distinguished by sets of the form
\eqref{eq:valuative spectrum patch basis}, hence also by any collection of generators of $\calB$.
In particular, $\Spv(A)$ is $T_0$.
Next, define a map from $\Spv(A)$ to $\{0,1\}^{A \times A}$ by the formula
\[
v \mapsto (v_{a,b}), \qquad v_{a,b} = \begin{cases} 1 & \mbox{if $v(a) \geq v(b)$,} \\ 0 & \mbox{if $v(a) < v(b)$};
\end{cases}
\]
it is injective with image defined by closed conditions (see \cite[Proposition~2.2]{huber1} and Lemma~\ref{L:reified is spectral}).
Equip $\{0,1\}$ with the discrete topology and $\{0,1\}^{A \times A}$ with the product topology.
By Tikhonov's theorem, the subspace topology on $\Spv(A)$ 
is compact and the elements of $\calB$ form a basis of closed-open subsets.
Since the given topology on $\Spv(A)$ is $T_0$ and is generated
by a set of generators of $\calB$, the claim follows from
Lemma~\ref{L:extract spectral}.
(It also follows that the subspace topology on $\Spv(A)$ coincides with the patch topology.)
\end{proof}

\begin{lemma} \label{L:lift valuations on fields}
For any finite normal extension $\ell/k$ of fields, the fibers of the map $\Spv(\ell) \to \Spv(k)$ are nonempty, finite, and permuted transitively by $\Aut(\ell/k)$.
\end{lemma}
\begin{proof}
See \cite[Propositions VI.8.9--12]{bourbaki-ac}.
\end{proof}

\begin{lemma} \label{L:valuation product}
Let $k_1/k, k_2/k$ be extensions of fields.
\begin{enumerate}
\item[(a)]
The maps $\Spv(k_1) \to \Spv(k), \Spv(k_2) \to \Spv(k)$ are surjective.
\item[(b)]
The map $\Spv(k_1 \otimes_k k_2) \to \Spv(k_1) \times_{\Spv(k)} \Spv(k_2)$ is surjective
(but typically not injective; see Remark~\ref{R:valuation product not injective}).
\end{enumerate}
\end{lemma}
\begin{proof}
To prove (a), see \cite[\S 5]{vaquie}. 
To prove (b), we follow \cite[Exercice VI.2.2]{bourbaki-ac}.
Let $\frako_v, \frako_{v_1}, \frako_{v_2}$ be the valuation rings of $v, v_1, v_2$;
let $\frakp_v, \frakp_{v_1}, \frakp_{v_2}$ be the maximal ideals of $v,v_1,v_2$;
and put $R = \frako_{v_1} \otimes_{\frako_v} \frako_{v_2}$. 
By a standard argument \cite[Tag 0495]{stacks-project},
there exists a prime ideal $\frakp$ of $R$ lying over $\frakp_{v_1}$ and $\frakp_{v_2}$.
Now recall that a module over a valuation ring is flat if and only if it is torsion-free
\cite[Tag~0539]{stacks-project}; by base extension, the morphisms $\frako_{v_1} \to R$,
$\frako_{v_2} \to R$ are flat and thus satisfy the going-down theorem \cite[Tag 00HS]{stacks-project}.
That is, we may construct first a prime ideal $\frakp_1 \subseteq \frakp$ lying over $(0) \subset \frako_{v_1}$ and then a prime ideal $\frakq \subseteq \frakp_1$ lying over $(0) \subset \frako_{v_2}$. Any valuation ring dominating $(R/\frakq)_{\frakp/\frakq}$ then corresponds to a semivaluation on $k_1 \otimes_k k_2$ restricting to $v_1$ on $k_1$ and to $v_2$ on $k_2$.
\end{proof}

\begin{remark} \label{R:extend two valuations}
Lemma~\ref{L:valuation product}(b) is implicitly invoked several times in Huber's work
(see \cite[Lemma~3.9(i)]{huber2}, \cite[(1.1.14)(e)]{huber-book}); the argument given above was suggested by Huber (private communication). One can also give a proof in the style of \cite[Theorem~4.1]{huber1} using the model theory of algebraically closed valued fields (ACVF), as follows.
(See Remark~\ref{R:ACVF} for a related discusson.)

Fix $v_1 \in \Spv(k_1)$, $v_2 \in \Spv(k_2)$ which both restrict to
$v \in \Spv(k)$; we must exhibit a common overfield $k_3$ of $k_1$ and $k_2$ and an element $v_3 \in \Spv(k_3)$ mapping to $v_1 \in \Spv(k_1)$ and to $v_2 \in \Spv(k_2)$.
Suppose first that $k_1/k$ is a finite extension.
By Lemma~\ref{L:lift valuations on fields}, we are free to first replace $k_1$ and $k_2$ by suitable algebraic extensions; we may thus ensure that $k_1$ is normal over $k$ and that $k_2$ contains a subfield $k_1'$ isomorphic to $k_1$. In this case, Lemma~\ref{L:lift valuations on fields} implies the existence of an isomorphism $k_1 \cong k_1'$ compatible with valuations, proving the claim.

To check the general case, 
by Zorn's lemma, we may assume that $k_2 = k(x)$. By Lemma~\ref{L:lift valuations on fields} and the previous paragraph, we may further assume that both $k$ and $k_1$ are algebraically closed; by adjoining an extra transcendental, we may further assume that the valuations on $k,k_1, k_2$ are all nontrivial.
By quantifier elimination in ACVF
(e.g., see \cite{hhm}), for any rational functions $f_1,\dots,f_n \in k(T)$, there exist a field extension $k_3$ of $k_1$, a Krull extension $v_3$ on $k_3$ restricting to $v_1$ on $k_1$, and an element $y \in k_3$ 
such that for $i=1,\dots,n$, 
we have $v_2(f_i(x)) \leq 1$ if and only if $v_3(f_i(y)) \leq 1$.

By Lemma~\ref{L:valuative is spectral}, $\Spv(k_1(x))$ is a spectral space
and hence is compact for the patch topology.
By the previous paragraph and the finite intersection property, there exists $v_3 \in \Spv(k_1(x))$ restricting to $v_1$ on $k_1$ and to $v_2$ on $k(x) \cong k_2$.
This proves the claim.
\end{remark}

\begin{remark} \label{R:valuation product not injective}
As one may infer from the analogy with schemes, the map in Lemma~\ref{L:valuation product}(b) is not injective. For example, if $k_1 = k(x)$, $k_2 = k(y)$ with $x,y$ transcendental over $k$, then the set
$\Spv(k_1 \otimes_k k_2)$ contains the trivial valuation on $k_1 \otimes_k k_2$, but it also contains many nontrivial semivaluations which restrict trivially to $k_1, k_2$.
One of these may be constructed by restricting the trivial valuation along the map $k_1 \otimes_k k_2 \to k_1$ which acts on $k_1$ as the identity map and on $k_2$ as the $k$-linear identification $k_2 \cong k_1$ mapping $y$ to $x$.
\end{remark}

\section{Ordinary adic spectra}

We next recall Huber's construction of adic spectra and definition of adic spaces.

\begin{defn}
A \emph{linearly topologized ring} (or \emph{LT ring} for short)
is a topological ring $A$ admitting a neighborhood basis of $0$
consisting of additive subgroups. 
For $A$ an LT ring, a subset $B$ of $A$ is \emph{bounded}
if for each neighborhood $U$ of $0$ in $A$, there exists a neighborhood $V$ of $0$ in $A$
with $V \cdot B \subseteq U$. 
An element $a \in A$ is \emph{power-bounded} (resp.\ \emph{topologically nilpotent}) if the sequence $a,a^2, \dots$ is bounded (resp.\ converges to 0).
The set $A^{\circ}$ of power-bounded elements is a subring of $A$;
the set $A^{\circ \circ}$ of topologically nilpotent elements is an ideal of $A^{\circ}$.
\end{defn}

\begin{defn} \label{D:LT tensor product}
For $A \to B, A \to C$ two continuous homomorphisms of LT rings, the tensor product $B \otimes_A C$ may be topologized in such a way that subgroups of the form $U \otimes V$, in which $U,V$ are additive subgroups which are neighborhoods of 0 in $B,C$, form a neighborhood basis of 0. Note that even if $A,B,C$ are all Hausdorff, $B \otimes_A C$ need not be.
\end{defn}

\begin{defn} \label{D:rational subspace LT}
Let $A$ be an LT ring. A \emph{rational subspace} of $\Spv(A)$
is a subset of the form
\begin{equation} \label{eq:rational subspace LT}
\{v \in \Spv(A): v(f_i) \leq v(f_0) \neq 0 \quad (i=1,\dots,n)\}
\end{equation}
for some $f_0,\dots,f_n \in A$ such that $f_1,\dots,f_n$ generate an open ideal of $A$.
By Lemma~\ref{L:valuative is spectral},
any rational subspace is quasicompact and open.
Note that any pairwise intersection of rational subspaces is again a rational subspace:
\begin{gather*}
\{v \in \Spv(A): v(f_i) \leq v (f_0) \neq 0 \quad (i=0,\dots,n)\} \\
\bigcap
\{v \in \Spv(A): v(g_j) \leq v(g_0) \neq 0 \quad (j=0,\dots,m)\} \\
=
\{v \in \Spv(A): v(f_i g_j) \leq v(f_0 g_0) \neq 0 \quad (i=0,\dots,n; j=0,\dots,m)\}.
\end{gather*}
One does not change the definition of a rational subspace if one requires only that $f_0, f_1,\dots,f_n$ generate an open ideal: if $n>0$,
one may add the condition $v(f_0) \leq v(f_0) \neq 0$ for free; if $n=0$, we have the space $\Spv(A)$ itself, and so we may as well take $f_0 = 1$.
\end{defn}

\begin{remark}
The definition of a rational subspace of $\Spv(A)$ we are using is the one from \cite{huber-book}. The definition in \cite{huber2} is formally different, but again can be shown to lead to the same class of subspaces.
\end{remark}

\begin{defn}
An \emph{adic ring} is a topological ring $A$ admitting an ideal $I$
whose powers form a fundamental system of neighborhoods of 0.
Any ideal with this property is called an \emph{ideal of definition} of $A$.

An \emph{f-adic ring} is a topological ring $A$ admitting an open subring $A_0$ which
is adic with a finitely generated ideal of definition. Any such subring $A_0$ is called a \emph{ring of definition} of $A$. Note that any f-adic ring is LT,
and the tensor product of f-adic rings (in the sense of Definition~\ref{D:LT tensor product}) is again f-adic.
\end{defn}

\begin{defn} \label{D:Tate f-adic ring}
An f-adic ring $A$ is \emph{Tate} if it contains a topologically nilpotent unit.
In this case, any open ideal is trivial; that is, if $f_1,\dots,f_n$ generate an open ideal of $A$, then for any $v \in \Spv(A)$, the quantities $v(f_1),\dots,v(f_n)$ cannot all vanish (e.g., see Corollary~\ref{C:units} below).
One consequence of this is that \eqref{eq:rational subspace LT} can be rewritten as
\[
\{v \in \Spv(A): v(f_i) \leq v(f_0) \quad (i=1,\dots,n)\}.
\]
This modification is needed to compare the concept of a rational subspace of $\Spv(A)$ with analogous concepts, such as that of a rational subspace of an affinoid space in rigid analytic geometry (as in \cite{BGR}).
\end{defn}

\begin{defn}
Let $A$ be an f-adic ring. A semivaluation $v \in \Spv(A)$ is \emph{continuous} if for every $\gamma \in \Gamma_v$,
there exists a neighborhood $U$ of $0$ in $A$ such that $v(u) < \gamma$ for all $u \in U$.
Let $\Cont(A)$ be the subspace of $\Spv(A)$ consisting of continuous semivaluations.
\end{defn}

The space $\Cont(A)$ does not naturally embed as a closed subspace for the patch
topology in a known spectral space like $\Spv(A)$. Nonetheless, using the finite generation of an ideal
of definition, one can prove the following.
\begin{theorem} \label{T:continuous valuation}
For any f-adic ring $A$, $\Cont(A)$ is a spectral space.
\end{theorem}
\begin{proof}
See \cite[Corollary~3.2]{huber1}.
\end{proof}

\begin{defn}
For $A$ an f-adic ring, a \emph{ring of integral elements} of $A$ is a subring $B$ of $A^{\circ}$ which is open in $A$
and integrally closed in $A$.
An \emph{affinoid f-adic ring} is a pair $(A^{\rhd}, A^+)$ in which $A^{\rhd}$ is an f-adic ring and $A^+$ 
is a ring of integral elements of $A^{\rhd}$.
A \emph{morphism} $(A^{\rhd}, A^+) \to (B^{\rhd}, B^+)$ of affinoid f-adic rings 
consists of a morphism $A^{\rhd} \to B^{\rhd}$ of topological rings carrying $A^+$ into $B^+$.

For $(A^{\rhd}, A^+) \to (B^{\rhd}, B^+)$, $(A^{\rhd}, A^+) \to (C^{\rhd}, C^+)$ two morphisms of affinoid f-adic rings, we define the \emph{tensor product}
$(B^{\rhd}, B^+) \otimes_{(A^{\rhd}, A^+)} (C^{\rhd}, C^+)$
to be the pair $(D^{\rhd}, D^+)$ in which $D^{\rhd} = B^{\rhd} \otimes_{A^{\rhd}} C^{\rhd}$
(in the sense of Definition~\ref{D:LT tensor product})
and $D^+$ is the integral closure of the image of $B^+ \otimes_{A^+} C^+$ in $D^{\rhd}$.
\end{defn}

\begin{defn}
Let $(A^{\rhd}, A^+)$ be an affinoid f-adic ring. The \emph{adic spectrum} of $(A^{\rhd}, A^+)$
is the subspace $\Spa(A^{\rhd},A^+)$ of $\Cont(A^{\rhd})$ consisting of those $v$ for which
$v(a) \leq 1$ for all $a \in A^+$. A \emph{rational subspace} of $\Spa(A^{\rhd},A^+)$
is the intersection of $\Spa(A^{\rhd}, A^+)$ with a rational subspace of $\Spv(A^{\rhd})$.
Any morphism $\varphi: (A^{\rhd}, A^+) \to (B^{\rhd}, B^+)$ defines a continuous map
$\varphi^*: \Spa(B^{\rhd}, B^+) \to \Spa(A^{\rhd}, A^+)$ by restriction.
\end{defn}

\begin{theorem} \label{T:adic spectrum is spectral}
For any affinoid f-adic ring $(A^{\rhd}, A^+)$, the space 
$\Spa(A^{\rhd},A^+)$ is spectral. Moreover, the rational subspaces  form a basis of the topology
of $\Spa(A^{\rhd},A^+)$ consisting of quasicompact open subsets.
\end{theorem}
\begin{proof}
See \cite[Theorem~3.5(i,ii)]{huber1}. The first assertion can also be deduced from
Theorem~\ref{T:continuous valuation} using Corollary~\ref{C:closed in spectral}.
\end{proof}

\begin{lemma} \label{L:nonempty adic spectrum}
For $A = (A^{\rhd}, A^+)$ an affinoid f-adic ring, $\Spa(A)$ is empty if and only if 
$0$ is dense in $A^{\rhd}$. (In particular, this condition does not depend on $A^+$.)
\end{lemma}
\begin{proof}
See \cite[Proposition~3.6]{huber1} or Theorem~\ref{T:Gel'fand spectrum} below.
\end{proof}

\begin{theorem} \label{T:tensor product to fiber product}
For $(A^{\rhd}, A^+) \to (B^{\rhd}, B^+)$, $(A^{\rhd}, A^+) \to (C^{\rhd}, C^+)$
two morphisms of affinoid f-adic rings
and $(D^{\rhd}, D^+) = (B^{\rhd}, B^+) \otimes_{(A^{\rhd}, A^+)} (C^{\rhd}, C^+)$,
the map
\[
\Spa(D^{\rhd}, D^+) \to \Spa(B^{\rhd}, B^+) \times_{\Spa(A^{\rhd}, A^+)} \Spa(C^{\rhd}, C^+)
\]
is surjective.
\end{theorem}
\begin{proof}
Given $v_1 \in \Spa(B^{\rhd}, B^+)$, $v_2 \in \Spa(C^{\rhd}, C^+)$ restricting to
$v \in \Spa(A^{\rhd}, A^+)$, Lemma~\ref{L:valuation product}
produces $v_3 \in \Spv(D^{\rhd})$ restricting to $v_1$ on $B^{\rhd}$ and to $v_2$ on $C^{\rhd}$. It is immediate that $v_3(x) \leq 1$ for all $x \in D^{\rhd}$, but not that $v_3$ is continuous. However, as in the proof of \cite[Lemma~3.9(i)]{huber2}, we may modify $v_3$ to obtain a continuous valuation by identifying certain infinitesimals with 0
(or see Definition~\ref{D:retract} below).
\end{proof}

\begin{defn} \label{D:Tate algebra}
Let $A$ be an f-adic ring.
Choose a ring of definition $A_0$ of $A$ and an ideal of definition $I$ of $A_0$.
We may view $A[T_1,\dots,T_n]$ as an f-adic ring with $A_0[T_1,\dots,T_n]$ as a ring of definition and $I A_0[T_1,\dots,T_n]$ as an ideal of definition; this does not depend on the choices of $A_0$ and $I$. 
Taking the completion yields another f-adic ring denoted $A\{T_1,\dots,T_n\}$ and called the \emph{Tate algebra} over $A$ in the variables $T_1,\dots,T_n$.
\end{defn}

\begin{defn} \label{D:rational subspace adic}
Let $(A^{\rhd}, A^+)$ be an affinoid f-adic ring.
Consider a rational subspace $U$ of $\Spv(A^{\rhd})$ defined by parameters $f_0,\dots,f_n \in A^{\rhd}$ as in \eqref{eq:rational subspace LT}.
Let $B^{\rhd}$ be the quotient of $A^{\rhd}\{T_1,\dots,T_n\}$ by the completion of the ideal $(f_0 T_1 - f_1, \dots, f_0 T_n - f_n)$.
Let $B^+$ be the completion of the integral closure of the image of $A^+[T_1,\dots,T_n]$ in $B^{\rhd}$.
We now have a morphism $(A^{\rhd}, A^+) \to (B^{\rhd}, B^+)$ of affinoid f-adic rings;
by Lemma~\ref{L:rational subspace} below, this construction depends only on the original rational subspace $U$ and not on the defining parameters.
\end{defn}

\begin{lemma} \label{L:rational subspace}
Retain notation as in Definition~\ref{D:rational subspace adic}.
\begin{enumerate}
\item[(a)]
The morphism $(A^{\rhd}, A^+) \to (B^{\rhd}, B^+)$ is initial among morphisms
$(A^{\rhd}, A^+) \to (C^{\rhd}, C^+)$ for which $C^{\rhd}$ is complete and the image of $\Spa(C^{\rhd}, C^+)$ in 
$\Spa(A^{\rhd},A^+)$ is contained in $U$.
\item[(b)]
The induced map $\Spa(B^{\rhd}, B^+) \to U$ is a homeomorphism. More precisely, the rational subspaces of $\Spa(B^{\rhd}, B^+)$ correspond to the rational subspaces of $\Spa(A^{\rhd},A^+)$ contained in $U$.
\end{enumerate}
\end{lemma}
\begin{proof}
See \cite[Lemma~1.5]{huber2}.
\end{proof}

\begin{defn}
A \emph{locally valuation-ringed space}, or \emph{locally v-ringed space} for short, is a locally ringed space $(X, \calO_X)$ equipped with the additional data of, for each $x \in X$, a valuation $v_x$ on the local ring $\calO_{X,x}$. A morphism of locally v-ringed spaces $f: X \to Y$ is a morphism of locally ringed spaces
with the property that for each $x \in X$ mapping to $y \in Y$, the restriction of $v_x$ along the map $\calO_{Y,y} \to \calO_{X,x}$ is equal to $v_y$.
\end{defn}

\begin{defn}
Let $(A^\rhd, A^+)$ be an affinoid f-adic ring. 
The \emph{structure presheaf} on $X = \Spa(A^{\rhd}, A^+)$ is the presheaf $\calO$ assigning to each open subset $U$ the inverse limit of $B^{\rhd}$ as
$(A^{\rhd}, A^+) \to (B^{\rhd}, B^+)$ runs over all morphisms representing rational subspaces of $X$ contained in $U$.
The stalks of $\calO$ are local rings \cite[Proposition~1.6]{huber2}.

We say that $(A^{\rhd}, A^+)$ is \emph{sheafy} if the presheaf $\calO$ is in fact a sheaf; in particular, $A^{\rhd}$ must be complete.
In this case, $(X, \calO)$ is a locally ringed space, which we promote to a locally v-ringed space as follows: for $x \in X$ corresponding to $v \in \Spv(A)$, let $v_x$ be the continuous extension of $v$ to $\calO_{X,x}$.
Any locally v-ringed space of this form is called an \emph{affinoid adic space}.
A locally v-ringed space which is covered by open subspaces which are affinoid adic spaces is called an \emph{adic space}.
\end{defn}

Unfortunately, the sheafy condition is not always satisfied; see \cite{buzzard-verberkmoes,mihara} for counterexamples. Two important classes where it is satisfied are described by the following results of Huber (in case (a)) and Buzzard--Verberkmoes (in case (b)).

\begin{theorem} \label{T:sheafy conditions}
Suppose that $A^{\rhd}$ is Tate and that at least one of the following conditions holds.
\begin{enumerate}
\item[(a)]
The ring $A^{\rhd}$ is \emph{strongly noetherian}: for each nonnegative integer $n$, the ring
$A^{\rhd}\{T_1,\dots,T_n\}$ is noetherian.
(This case includes classical affinoid algebras; see Example~\ref{exa:affinoid algebras}.)
\item[(b)]
The pair $(A^{\rhd}, A^+)$ is \emph{stably uniform}: for every morphism $(A^{\rhd}, A^+) \to 
(B^{\rhd}, B^+)$ representing a rational subspace of $\Spa(A^{\rhd}, A^+)$,
$B^{\rhd,\circ}$ is open in $B^{\rhd}$.
\end{enumerate}
Then $(A^{\rhd}, A^+)$ is sheafy.
\end{theorem}
\begin{proof}
For (a), see \cite[Theorem~2.2(b)]{huber2}. For (b), see \cite[Theorem~7]{buzzard-verberkmoes}.
\end{proof}

\begin{remark} \label{R:preadic}
There is a process to attach ``spaces'' to affinoid f-adic rings which are not sheafy, but this requires a more abstract approach as originally described by
Scholze and Weinstein \cite{scholze-weinstein}.
See also \cite[\S 8.2]{kedlaya-liu1}.
\end{remark}

\begin{remark} \label{R:analytic adic}
Huber declares an adic space to be \emph{analytic} if it is covered by the adic spectra of affinoid f-adic rings which are not only sheafy, but also Tate. This extra restriction fails in some natural classes of examples (e.g., adic spaces associated to ordinary schemes or formal schemes), but is needed in order to make many classical arguments of rigid analytic geometry carry over to the setting of adic spaces. One pleasant feature of reified adic spaces is that there admit no analogue of the analytic condition; the role played by topologically nilpotent units is taken over by reifications.
\end{remark}

\section{Gel'fand spectra}
\label{sec:gelfand}

We next introduce the class of normed rings and describe Berkovich's construction of the Gel'fand spectrum of a normed ring.

\begin{defn}
A \emph{seminormed ring} (resp.\ a \emph{normed ring}) is a ring $A$ equipped with a \emph{seminorm} (resp.\ \emph{norm}), i.e., a function $\left| \bullet \right|: A \to [0, +\infty)$
satisfying the following conditions.
\begin{enumerate}
\item[(a)]
We have $|0| = 0$ (resp.\ for all $x \in A$, $x=0$ if and only if $|x| =0$).
\item[(b)]
For all $x,y \in A$, we have $|x+y| \leq |x| + |y|$.
\item[(c)]
For all $x,y \in A$, we have $|xy| \leq |x| |y|$.
\end{enumerate}
We say that a (semi)normed ring $A$ is \emph{nonarchimedean} if the upper bound in (b) can be improved to
$\max\{|x|, |y|\}$. The \emph{trivial norm} on $A$ is the norm for which $\left|x \right| = 1$ for all nonzero $x \in A$.

The \emph{(semi)norm topology} on a nonarchimedean (semi)normed ring $A$ is the metric topology induced by the seminorm. For this topology, $A$ is an LT ring.
\end{defn}

\begin{defn}
A morphism $f: A \to B$ of nonarchimedean seminormed rings is \emph{bounded} if there exists $c >0$ such that for all $a \in A$, we have $\left| f(a) \right| \leq c \left|a \right|$. Any such morphism is continuous (but not conversely).
\end{defn}

\begin{defn}
A \emph{(nonarchimedean commutative) Banach ring} is a nonarchimedean normed ring which is separated and complete for the norm topology. For $A$ a Banach ring, a \emph{Banach algebra} over $A$ is a Banach ring $B$ equipped with a bounded homomorphism $A \to B$.
\end{defn}

\begin{defn} \label{D:ultrametric field}
A \emph{ultrametric field} is a Banach ring $F$ such that the underlying ring $F$ is a field and the norm is a Krull valuation (i.e., the inequality in (c) is an equality). Unless otherwise specified, we allow this definition to include the case of a trivial norm.
\end{defn}

\begin{remark} \label{R:ultrametric field}
In Definition~\ref{D:ultrametric field}, 
the second condition is needed because one can modify the norm on $F$ without changing the norm topology, in such a way that the resulting norm is not itself a Krull valuation, e.g., by taking the supremum of the norms corresponding to two different reifications of the same underlying valuation. (Compare \cite[Remark~8.7]{kedlaya-noetherian}.)
\end{remark}

\begin{remark} \label{R:norm on f-adic ring}
Any f-adic ring $A$ can be viewed as a nonarchimedean seminormed ring (topologized using the seminorm topology).
For example, 
let $A_0$ be a ring of definition, let $I$ be a finitely generated ideal of definition of $A_0$, pick $c \in (0,1)$,
and define $\left| \bullet \right|: A \to [0, +\infty)$ as follows.
\begin{itemize}
\item
For $a \in A_0$, set $|a| = c^{-n}$ for $n$ the smallest nonnegative integer such that $a \notin I^{n+1}$
if such an integer exists; otherwise, set $|a| = 0$.
\item
For $a \notin A_0$, set $|a| = c^n$ for $n$ the smallest positive integer such that
$a I^n \subseteq A_0$. Such an integer must exist because $A_0$ is open in $A$.
\end{itemize}
Beware that the equivalence class of this norm is not uniquely determined by the topology of $A$ (because of the possibility of varying $c$ and $I$); in particular, this construction does not define a functor from f-adic rings to nonarchimedean seminormed rings.

In the other direction, for $A$ a nonarchimedean seminormed ring viewed as an LT ring using the seminorm topology, it is not immediate that $A$ is an f-adic ring; the difficulty is to find an ideal of definition which is finitely generated.
One case where this is possible is when $A$ contains a topologically nilpotent unit $x$
(i.e., $A$ is \emph{Tate}),
by taking $A_0$ to be the subring of $x \in A$ for which $\left| x \right| \leq 1$
and $I$ to be the ideal $(x^n)$ for $n$ large enough so that $x^n \in A_0$; consequently, any such $A$ is a Tate f-adic ring. In particular, if
$A$ is a nonzero Banach algebra over an ultrametric field $F$ with nontrivial norm, 
any $x \in F$ with $0 <  |x| < 1$ is a topologically nilpotent unit.
\end{remark}

As remarked above, an f-adic ring cannot be viewed as a nonarchimedean seminormed ring in a canonical way.  However, we have the following result.
\begin{lemma} \label{L:field norm to topology}
Let $R$ be a Banach ring which is Tate.
Then the forgetful functor from Banach rings over $R$ to 
complete f-adic rings $A$ equipped with continuous maps $R \to A$ is an equivalence of categories.
\end{lemma}
\begin{proof}
By Remark~\ref{R:norm on f-adic ring}, the functor is essentially surjective; full faithfulness is a consequence of the Banach open mapping theorem (see \cite{henkel}).
\end{proof}

\begin{defn}
For $A \to B, A \to C$ two bounded homomorphisms of nonarchimedean seminormed rings, we view the tensor product $B \otimes_A C$ as a nonarchimedean seminormed ring by equipping it with the \emph{tensor product seminorm}: the value of the seminorm on $x \in B \otimes_A C$ is the infimum of $\max_i \{\left| y_i \right| \left| z_i \right|\}$ over all presentations $x = \sum_i y_i \otimes z_i$.
\end{defn}

\begin{lemma} \label{L:tensor product nonzero}
Let $F \to E$ be a bounded homomorphism of ultrametric fields. Then for any nonzero Banach algebra $A$ over $F$, $A \otimes_E F$ is nonzero and Hausdorff.
\end{lemma}
\begin{proof}
See \cite[Lemma~2.2.9]{kedlaya-liu1}.
\end{proof}

\begin{defn}
Let $\alpha$ be an $\RR$-valued semivaluation on $A$. We may then extend $\alpha$ to an $\RR$-valued Krull valuation on
$\Frac(A/\ker(\alpha))$. Completing with respect to this extension yields an ultrametric field, denoted $\calH(\alpha)$. Note that $\alpha$ can be recovered as the restriction along the natural map $A \to \calH(\alpha)$ of the valuation on $\calH(\alpha)$.
\end{defn}

For the remainder of \S\ref{sec:gelfand}, let $A$ be a nonarchimedean normed ring.
\begin{defn}
The \emph{Gel'fand spectrum} of $A$ is the set $\calM(A)$ of $\RR$-valued semivaluations on $A$ which are bounded above by $\left| \bullet \right|$,
equipped with the evaluation topology. The inclusion $\calM(A) \hookrightarrow {\RR}^A$
is a homeomorphism of $\calM(A)$ onto a compact subspace of ${\RR}^A$
\cite[Theorem~1.2.1]{berkovich1}.
\end{defn}

\begin{theorem} \label{T:Gel'fand spectrum}
The space $\calM(A)$ is nonempty if $\{0\}$ is not dense in $A$. Moreover,
\[
|a|_{\mathrm{sp}} := \lim_{n \to \infty} |a^n|^{1/n} = \sup\{\alpha(a): \alpha \in \calM(A)\} \qquad (a \in A);
\]
that is, the spectral seminorm equals the supremum seminorm.
\end{theorem}
\begin{proof}
See \cite[Theorem~1.2.1, Theorem~1.3.1]{berkovich1}.
\end{proof}
\begin{cor} \label{C:units}
An ideal $I$ of $A$  contains $1$ in its closure if and only if for each $\alpha \in \calM(A)$, there exists $a \in I$ with $\alpha(a) > 0$.
\end{cor}
\begin{proof}
If $I$ does not contain 1 in its closure, then the quotient seminorm on $A/I$ is nonzero, so
Theorem~\ref{T:Gel'fand spectrum} applies to produce $\alpha \in \calM(A)$ whose restriction to $I$ is zero.
\end{proof}

\begin{defn} \label{D:rational subspace gelfand}
A \emph{rational subspace} of $\calM(A)$
is one of the form
\begin{equation} \label{eq:rational subspace gelfand}
U = \{\alpha \in \calM(A): \alpha(f_i) \leq q_i \alpha(f_0) \neq 0 \quad (i=1,\dots,n)\}
\end{equation}
for some $f_0, \dots,f_n \in A$ such that $f_1,\dots,f_n$ generate the unit ideal
and some $q_1, \dots, q_n > 0$.
If it is possible to take $q_1 = \cdots = q_n = 1$,
we call the resulting set a \emph{strictly rational subspace} of $\calM(A)$.

As in Definition~\ref{D:rational subspace LT}, 
the intersection of two (strictly) rational subspaces is (strictly) rational: 
taking $q_0 = r_0 = 1$, we have
\begin{gather*}
\{\alpha \in \calM(A): \alpha(f_i) \leq q_i \alpha(f_0) \neq 0 \quad (i=0,\dots,n)\} \\
\bigcap
\{\alpha \in \calM(A): \alpha(g_j) \leq r_j \alpha(g_0) \neq 0 \quad (j=0,\dots,m)\} \\
=
\{\alpha \in \calM(A): \alpha(f_i g_j) \leq q_i r_j \alpha(f_0 g_0) \neq 0 \quad (i=0,\dots,n; j=0,\dots,m)\}.
\end{gather*}
As in Definition~\ref{D:Tate f-adic ring},
by Corollary~\ref{C:units} the space \eqref{eq:rational subspace gelfand} can also be written as
\[
\{\alpha \in \calM(A): \alpha(f_i) \leq q_i \alpha(f_0) \quad (i=1,\dots,n)\}.
\]
Hence any rational subspace of $\calM(A)$ is closed.
\end{defn}

\begin{remark} \label{R:rational subspace definition}
Note that in Definition~\ref{D:rational subspace gelfand}, we require that $f_1,\dots,f_n$ generate the unit ideal, not merely an open ideal. This means that in case $A$ is f-adic
(which is itself not automatic; see Remark~\ref{R:norm on f-adic ring}), 
there is a natural map $\calM(A) \to \Spv(A)$ mapping each $\RR$-valued semivaluation to its equivalence class, but the preimage of a rational subspace of $\Spv(A)$ is not necessarily a rational subspace of $\calM(A)$
unless $A$ is Tate (see Definition~\ref{D:Tate f-adic ring}).
\end{remark}

\begin{remark} \label{R:rational subspace nearby generators}
With notation as in Definition~\ref{D:rational subspace gelfand},
note that by compactness,
\[
c = \inf\{\alpha(f_0): \alpha \in U\} > 0.
\]
Choose $h_1,\dots,h_n \in A$ such that $f_1 h_1 + \dots + f_n h_n = 1$.
Then for any $f'_0,\dots,f'_n$ such that
\[
\left|f'_0 - f_0\right| < c, \qquad
\left|f'_i - f_i\right| < \min\{q_i c, |h_i|^{-1}\} \quad (i=1,\dots,n),
\]
we have $|f'_1 h_1 + \cdots + f'_n h_n - 1| < 1$, so
$f'_1,\dots,f'_n$ again generate the unit ideal in $A$, and
\[
U = \{\alpha \in \calM(A): \alpha(f'_i) \leq q_i \alpha(f'_0) \neq 0 \quad (i=1,\dots,n)\}.
\]
(Compare \cite[Lemma~3.10]{huber1} and \cite[Remark~2.4.7]{kedlaya-liu1}.)
\end{remark}

\section{Reified valuations}
\label{sec:reified valuations}

In order to bring the valuation-theoretic and norm-theoretic viewpoints into alignment, 
and to give an explicit relationship between the two in the case of affinoid algebras (Theorems~\ref{T:affinoid patch dense1} and~\ref{T:affinoid patch dense2}),
we describe
a variation on the theory of valuations in which scaling ambiguities are eliminated.
Much of the resulting analysis runs parallel to the analysis in \cite{huber1} cited above,
although with some key differences due to the change in the definition of rational subspaces (see Remark~\ref{R:rational subspace definition}).

\begin{defn}
Let $\RR^+$ denote the multiplicative monoid of positive real numbers.
A \emph{reified value group} is a value group $\Gamma$ equipped with an order-preserving homomorphism
$r: \RR^+ \to \Gamma$.

Let $A$ be a ring. A \emph{reified semivaluation} on $A$ is a semivaluation $v: A \to \Gamma_0$ for $\Gamma$ a reified value group. Given a semivaluation $v: A \to \Gamma_0$,
we will refer to the extra data of an order-preserving homomorphism
$r: \RR^+ \to \Gamma$ as a \emph{reification} of $v$.

For $v$ a reified semivaluation, let $\Gamma_v$ be the subgroup of $\Gamma$ generated by $\RR^+$ and the nonzero images of $A$, viewed as a reified value group.
Two reified semivaluations $v_1, v_2$ on $A$ are \emph{equivalent} if 
there exists an isomorphism $i: \Gamma_{v_1} \cong \Gamma_{v_2}$ of reified value groups 
(which we also view as an isomorphism $i: \Gamma_{v_1,0} \cong \Gamma_{v_2,0}$)
such that $i \circ v_1 = v_2$. 
\end{defn}

\begin{defn} 
Let $A$ be a ring. The \emph{reified valuative spectrum} of $A$,
denoted $\Sprv(A)$, is the set of equivalence classes of reified semivaluations on $A$,
equipped with the topology
generated by sets of the form
\begin{equation} \label{eq:reified valuative spectrum basis}
\{v \in \Sprv(A): v(a) \leq q v(b) \neq 0\} \qquad (a,b \in A; q \in \RR^+).
\end{equation}
Again, if we let $\calB$ be the Boolean algebra generated by the basic open sets as in \eqref{eq:reified valuative spectrum basis}, then $\calB$ is also generated by the sets of the form
\begin{equation} \label{eq:reified valuative patch basis}
\{v \in \Sprv(A): v(a) \leq q v(b)\} \qquad (a,b \in A; q \in \RR^+).
\end{equation}
\end{defn}

\begin{remark} \label{R:reified lift}
There is a natural projection $\Sprv(A) \to \Spv(A)$ forgetting reifications,
which is surjective for trivial reasons: given any semivaluation $v: A \to \Gamma_0$, 
we may form an equivalent semivaluation by enlarging the value group to $\Gamma \times \RR^+$ (or $\RR^+ \times \Gamma$) ordered lexicographically.
For a more refined statement along the same lines, see Lemma~\ref{L:lift reification}.
\end{remark}

The analogue of the fact that the equivalence class of a valuation is determined by its order relation is the following.
\begin{lemma} \label{L:reified order relation} 
Let $A$ be a ring.
Define a map from $\Sprv(A)$ to $\{0,1\}^{A \times A \times \RR^+}$ by the formula
\[
v \mapsto (v_{a,b,q}), \qquad 
v_{a,b,q} = \begin{cases} 1 & \mbox{if $v(a) \geq qv(b)$,} \\ 0 & \mbox{if $v(a) < qv(b)$}. \end{cases}
\]
Then this map is injective and its image is cut out by the following closed conditions
(writing $a,b,c,d$ for arbitrary elements of $A$ and $q,r$ for arbitrary elements  of $\RR^+$):
\begin{itemize}
\item[(i)]
$v_{a,a,1} = 1$;
\item[(ii)]
if $v_{a,b,q} = 0$, then $v_{b,a,1/q} = 1$;
\item[(iii)]
if $v_{a,b,q} = v_{b,c,r} = 1$, then $v_{a,c,qr} = 1$;
\item[(iv)]
if $v_{a,b,q} = v_{c,d,r} = 1$, then $v_{ac,bd,qr} = 1$;
\item[(v)]
$v_{0,1,q} = 0$;
\item[(vi)]
if $q > 1$, then $v_{1,1,q} = 0$;
\item[(vii)]
if $v_{a,b,q} = v_{a,c,q} = 1$, then $v_{a,b+c,q} = 1$;
\item[(viii)]
if $v_{ac,bc,q} = 1$ and $v_{0,c,1} = 0$, then $v_{a,b,q} = 1$;
\item[(ix)] 
$v_{a,0,q} = 1$. (This is a consequence of (i), (ii), (iv), (v).)
\end{itemize}
\end{lemma}
\begin{proof}
Each condition is evidently closed and satisfied on the image of $\Sprv(A)$. 
Conversely, suppose that the tuple $(v_{a,b,q})$ belongs to the image.
We reconstruct the corresponding reified valuation $v$ as follows.

We first reconstruct the kernel of $v$. Put $\frakp_v := \{a \in A: v_{0,a,1} = 1\}$; this is an ideal (by (i), (iv), (vii)) which is proper (by (v)) and prime (by (viii)).

We next reconstruct the reified divisibility relation.
Put $S := A \times (A \setminus \frakp_v) \times \RR^+$. We define a binary relation $\leq$ on $S$ by declaring that 
\[
(a,b,q) \leq (c,d,r) \Leftrightarrow v_{bc,ad,q/r} = 1.
\]
By (i), the relation $\leq$ is reflexive.
We next check that for $(a,b,q) \in S$,
\begin{align} \label{eq:zero condition}
a \in \frakp_v &\Leftrightarrow (a,b,q) \leq (c,d,r) \mbox{ for all } (c,d,r) \in S \\
&\Leftrightarrow (a,b) \leq (c,d,r) \mbox{ for some } (c,d,r) \in S
\mbox{ with } c \in \frakp_v.
\nonumber
\end{align} 
On one hand, if $a \in \frakp_v$,
then $ad \in \frakp_v$ and so $v_{0,ad,1} = 1$; by (ix), $v_{bc,0,q/r} = 1$;
by (iii), we have $(a,b,q) \leq (c,d,r)$. On the other hand, if
$(a,b,q) \leq (c,d,r)$ and $c \in \frakp_v$, then $v_{bc,ad,q/r} = 1$
and $bc \in \frakp_v$, so by (iii), $v_{0,ad,q/r} = 1$.
By (ix), $v_{0,1,r/q} = 1$; by (iv), $ad \in \frakp_v$. Since $\frakp_v$ is prime, $a \in \frakp_v$.

To check that $\leq$ is transitive, we assume $(a,b,q) \leq (c,d,r) \leq (e,f,s)$ and distinguish two cases. If $c \in \frakp_v$, then $a \in \frakp_v$ by \eqref{eq:zero condition}
and so $(a,b,q) \leq (e,f,s)$. If $c \notin \frakp_v$, then $v_{bc,ad,q/r} = 1$
and $v_{de,cf,r/s} = 1$, so by (iv), $v_{bcde,acdf,q/s} = 1$.
Since $cd \notin \frakp_v$, by (viii) we have $(a,b,q) \leq (e,f,s)$.

We next reconstruct the underlying reified value group.
Define an equivalence relation equating $(a,b,q), (c,d,r) \in S$ whenever
$(a,b,q) \leq (c,d,r)$ and $(c,d,r) \leq (a,b,q)$. 
Let $\Gamma_{v,0}$ be the set of equivalence classes.
Let $0 \in \Gamma_{v,0}$ be the class of $(0,1,1)$; by \eqref{eq:zero condition}, this consists of those $(a,b,q)$ with $a \in \frakp_v$. 
Equip $S$ with the binary operation $\cdot$ given by
\[
(a,b,q) \cdot (c,d,r) \mapsto (ac,bd,qr);
\]
for this operation, $S$ is a commutative monoid with identity element $(1,1,1)$. By (iv), $\cdot$ is monotonic with respect to $\leq$; it thus induces a monoid structure on $S$. For $(a,b,q) \in S$ with $a \notin \frakp_v$, we have $ab \notin \frakp_v$ and so $(b,a,1/q) \in S$; by (i), $(a,b,q)$ and $(b,a,1/q)$ define inverse classes in $\Gamma_{v,0}$. If we set $\Gamma_v = \Gamma_{v,0} \setminus \{0\}$ and define the map $r: \RR^+ \to \Gamma_{v}$ taking $q$ to the class of $(1,1,q)$, it follows that $\Gamma_v$ is a reified value group with identity element the class of $(1,1,1)$ and associated pointed commutative monoid  $\Gamma_{v,0}$.

To conclude, let $v: A \to \Gamma_{v,0}$ be the function taking $a$ to the class of $(a,1,1)$.
By (vii) and (iv), $v$ is a reified semivaluation whose image in 
$A \times A \times \RR^+$ is the tuple $(v_{a,b,q})$;
by (vi), $v$ is unique for this property.
\end{proof}

This gives rise to  the following analogue of Lemma~\ref{L:valuative is spectral}, with a similar proof.
\begin{lemma} \label{L:reified is spectral}
For any ring $A$, the space $\Sprv(A)$ is spectral, and sets of the form 
\eqref{eq:reified valuative spectrum basis} are quasicompact and open.
\end{lemma}
\begin{proof}
Since sets of the form \eqref{eq:reified valuative patch basis} clearly separate points, the space  $\Sprv(A)$ is $T_0$.
On the other hand, by Lemma~\ref{L:reified order relation}, we may identify $\Sprv(A)$
with a closed subspace of the compact space $\{0,1\}^{A \times A \times \RR^+}$ (for the discrete topology on $\{0,1\}$) in such a way that spaces of the form 
\eqref{eq:reified valuative spectrum basis} are closed-open. 
By Lemma~\ref{L:extract spectral}, $\Sprv(A)$ is spectral
and subsets of the form 
\eqref{eq:reified valuative spectrum basis} are quasicompact and open.
\end{proof}

\begin{lemma} \label{L:lift reification}
Let $\ell/k$ be an extension of fields. Then the map
\[
\Sprv(\ell) \to \Spv(\ell) \times_{\Spv(k)} \Sprv(k)
\]
is surjective (but typically not injective).
\end{lemma}
\begin{proof}
We may reduce to the case where $\ell = k(x)$ for some $x \in \ell$.
Let $v_1$ be a valuation on $\ell$ restricting to the valuation $v$ on $k$.
If the inclusion $\Gamma_v \to \Gamma_{v_1}$
induces an isomorphism $\Gamma_v \otimes_{\ZZ} \QQ \to \Gamma_{v_1} \otimes_{\ZZ} \QQ$,
then any reification of $v$ induces a unique reification of $v_1$. This already suffices to treat the case where $x$ is algebraic over $k$.

If $x$ is transcendental over $k$, using the previous paragraph we may reduce to the case where $k$ is itself algebraically closed and 
$\Gamma_{v_1} \neq \Gamma_v$ (the latter group being divisible).
Since $k$ is algebraically closed, there must exist $c \in k, d \in k^\times$ for which
$v_1(cx-d) < 1$ and $v_1(cx-d) \notin \Gamma_v$. For any $P = \sum_{n \geq 0} P_n T^n \in k[T]$, we have
\[
v_1(P(cx-d)) = \max_n \{v(P_n)  v_1(cx-d)^n\}
\]
since the nonzero terms in the maximum are pairwise distinct. It follows that as abstract groups we have
$\Gamma_{v_1} = \Gamma_v \oplus v_1(cx-d)^{\ZZ}$.

If $v$ is trivial, then the claim is equally trivial, so we may assume hereafter that $v$ is nontrivial. Let $\tilde{v}$ be a reification of $v$.
Let $S_-$ (resp.\ $S_+$) be the set of $t \in \RR$ for which there exists $y \in k$
such that $v_1(cx-d) > v_1(y)$ and $\tilde{v}(y) \geq t$ (resp.\ $v_1(cx-d) < v_1(y)$ and $\tilde{v}(y) \leq t$). These sets have the following properties.
\begin{itemize}
\item
The set $S_-$ is down-closed and contains 0.
\item
The set $S_+$ is up-closed and contains 1.
\item
The sets $S_-, S_+$ are disjoint.
\item
The supremum of $S_-$ equals the infimum of $S_+$. We denote the common value by $s$.
\end{itemize}
Note that it is possible to choose $\epsilon \in \{-1,0,1\}$ subject to the following conditions.
\begin{itemize}
\item
If $s \in S_-$, then $\epsilon \neq -1$.
\item
If $s \in S_+$, then $\epsilon \neq +1$.
\item
If $s \in \{0,1\}$, then $\epsilon \neq 0$.
\end{itemize}
Let $\Gamma$ be the lexicographic product $\Gamma_{\tilde{v}} \times \RR^+$ viewed as a reified value group via the reification on the first factor. We then obtain
an embedding of $\Gamma_{v_1}$ into $\Gamma$ taking $v_1(cx-d)$ to $(s, e^\epsilon)$, yielding a reification of $v_1$ as desired.
\end{proof}

We have the following analogue of Lemma~\ref{L:valuation product}.
\begin{lemma} \label{L:valuation product2}
Let $k_1/k, k_2/k$ be extensions of fields.
\begin{enumerate}
\item[(a)]
The maps $\Sprv(k_1) \to \Sprv(k), \Sprv(k_2) \to \Sprv(k)$ are surjective.
\item[(b)]
The map $\Sprv(k_1 \otimes_k k_2) \to \Sprv(k_1) \times_{\Sprv(k)} \Sprv(k_2)$ is surjective.
\end{enumerate}
\end{lemma}
\begin{proof}
This is immediate from Lemma~\ref{L:valuation product} and
Lemma~\ref{L:lift reification}.
\end{proof}

For the remainder of \S \ref{sec:reified valuations}, fix a nonarchimedean normed ring $A$. We first cut down the space $\Sprv(A)$ by imposing some interaction between the seminorm on $A$ and the reifications.

\begin{defn}
For $r \in \RR^+$,
let $A^{\circ,r}$ (resp.\ $A^{\circ \circ,r}$) be the set of $a \in A$ such that the sequence $\{r^{-n} \left| a^n \right|\}_{n=1}^\infty$ is bounded (resp.\ converges to 0).
Note that $A^{\circ,1} = A^\circ$ and $A^{\circ \circ, 1} = A^{\circ \circ}$;
more generally, if $a \in A^{\circ,r}$, then $\left| a \right|_{\mathrm{sp}} \leq r$, but not conversely. By contrast, if $a \in A^{\circ \circ,r}$ then $\left| a \right|_{\mathrm{sp}} < r$ \emph{and conversely}: if $\left| a \right|_{\mathrm{sp}} < r$,
then there exist a positive integer $m$ and a value $c \in (0,1)$ such that $\left| a^m \right| \leq c^m r^m$,
so for all $n \geq 0$ we have
\[
r^{-n} \left| a^n \right| \leq c^{\lfloor n/m \rfloor} \max\{r^{-i} \left| a^i \right|: i=0,\dots,m-1\}
\]
and so $a \in A^{\circ \circ,r}$.
Following Temkin \cite{temkin-local2}, we define the graded ring
\[
\Gr A = \bigoplus_{r>0} \Gr^r A, \qquad \Gr^r A = A^{\circ,r}/A^{\circ \circ,r}.
\]
\end{defn}

\begin{defn}
A reified semivaluation $v: A \to \Gamma_0$ on $A$ is \emph{bounded}
if for all $r \in \RR^+$ and $a \in A^{\circ \circ,r}$, we have $v(a) \leq r$. 
Let $\Sprb(A)$ be the subspace of $\Sprv(A)$ consisting of the equivalence
classes of bounded reified semivaluations.
\end{defn}

The analogue of continuity for reified valuations is the following condition.
\begin{defn}
A reified semivaluation $v: A \to \Gamma_0$ on $A$ is \emph{commensurable}
if it bounded and for all $\gamma \in \Gamma_v$, there exists $r \in \RR^+$ such that $r \leq \gamma$.
Let $\Comm(A)$ be the subspace of $\Sprb(A)$ consisting of the equivalence classes of commensurable reified semivaluations. 
In case the underlying topological ring of $A$ is an f-adic ring, the map $\Sprv(A) \to \Spv(A)$ forgetting reifications induces a map $\Comm(A) \to \Cont(A)$.
\end{defn} 

This construction is designed to eliminate certain infinitesimals.
\begin{example} \label{exa:not commensurate}
Put $A = \QQ_p[T]$. Put $\Gamma = u^{\ZZ} \times \RR^+$ with the lexicographic ordering
and define the valuation $v: A \to \Gamma_0$ by setting 
\[
v\left(\sum_i a_i T^i \right) = \max_i \{ (u^{-i}, p^{-v_p(a_i)})\}.
\]
Then $v \in \Sprv(A, \left| \bullet \right|)$ but $v \notin \Comm(A)$.
\end{example}

\begin{defn} \label{D:real projection}
Since every $\RR$-valued semivaluation is commensurable,
there is a natural inclusion $\calM(A) \to \Comm(A)$.
There is also a natural map $\Sprb(A) \to \calM(A)$ taking $v \in \Sprb(A)$ to the map $\alpha$ defined as follows: 
for $a \in A$, $\alpha(a)$ is the infimum of all $r \in \RR^+$ for which $\alpha(a) \leq r$.
The composition $\calM(A) \to \Comm(A) \to \Sprb(A) \to \calM(A)$ is the identity map.
\end{defn}

\begin{defn}
For $v \in \Sprb(A)$ mapping to $\alpha \in \calM(A)$ as in Definition~\ref{D:real projection}, put $\calH(v) = \calH(\alpha)$.
We may then extend $v$ by continuity to a Krull valuation on $\calH(v)$,
from which we may recover $v$ by pullback along $A \to \calH(v)$.
\end{defn}

\begin{defn} \label{D:retract}
For $\Gamma$ a reified value group,
let $\overline{\Gamma}$ be the subgroup of $\Gamma$ consisting of those $\gamma \in 
\Gamma$ such that $r \leq \gamma \leq r^{-1}$ for some $r \in \RR^+$.
(The lower bound on $r$ is the one we will need in the construction; the upper bound is there to ensure that we get a subgroup.)
For $v \in \Sprb(A)$, the function $r(v): A \to \overline{\Gamma}_0$ defined by
\[
r(v)(a) = \begin{cases} v(a) & \mbox{if $a \in \overline{\Gamma}$,} \\
0 & \mbox{if $a \notin \overline{\Gamma}$,}
\end{cases}
\]
is a commensurable reified semivaluation. We thus obtain a map $r: \Sprb(A) \to \Comm(A)$ of which the inclusion $\Comm(A) \to \Sprb(A)$ is a section.
\end{defn}

\begin{defn} \label{D:rational subspace reified}
By analogy with $\calM(A)$ (but not with $\Spv(A)$; see Remark~\ref{R:rational subspace definition}), we define a \emph{rational subspace} of $\Sprv(A)$
to be a subset of the form
\begin{equation} \label{eq:rational subspace reified}
\{v \in \Sprv(A): v(f_i) \leq q_i v(f_0) \neq 0 \quad (i=1,\dots,n)\}
\end{equation}
for some $f_0, \dots,f_n \in A$ such that $f_1,\dots,f_n$ generate the unit ideal
and some $q_1, \dots, q_n > 0$. 
As in Definition~\ref{D:rational subspace gelfand}, we can rewrite \eqref{eq:rational subspace reified} as
\begin{equation} \label{eq:rational subspace reified2}
\{v \in \Sprv(A): v(f_i) \leq q_i v(f_0) \quad (i=1,\dots,n)\}.
\end{equation}
Any rational subspace is quasicompact (by Lemma~\ref{L:reified is spectral}) and open.
A \emph{rational subspace} of $\Sprb(A)$ or $\Comm(A)$ is the intersection
of said space with a rational subspace of 
$\Sprv(A)$. 
\end{defn}

\begin{lemma} \label{L:strong rational subspaces}
The rational subspaces of $\Comm(A)$ form a basis for the topology. 
\end{lemma}
\begin{proof}
Choose any $v \in \Comm(A)$. For any $a,b \in A$ and $q>0$,
if the set
\[
U = \{w \in \Comm(A): w(a) \leq q w(b) \neq 0\}
\]
contains $v$, then there exists $r \in \RR^+$ such that $r \leq v(b)$,
so
\[
\{w \in \Comm(A): w(a) \leq q w(b) \neq 0, w(1) \leq (1/r) w(b) \neq 0\}
\]
is a rational subspace of $\Comm(A)$ containing $v$ and contained in $U$.
Since the set of rational subspaces is closed under finite intersections, this proves the claim.
\end{proof}

\begin{lemma} \label{L:retract}
For $r: \Sprb(A) \to \Comm(A)$ as in
Definition~\ref{D:retract}, 
for $U$ a rational subspace of $\Sprb(A)$, we have $r^{-1}(U \cap \Comm(A)) = U$.
\end{lemma}
\begin{proof}
Note that $r$ preserves order relations: if $v(a) \leq v(b)$, then
$r(v)(a) \leq r(v)(b)$. 
By expressing $U$ in the form \eqref{eq:rational subspace reified2}, we see that $r(U) \subseteq U \cap \Comm(A)$,
so $U \subseteq r^{-1}(U \cap \Comm(A))$. On the other hand, if $v \in \Sprb(A)$ satisfies
$r(v) \in U \cap \Comm(A)$, then $r(v)(f_0) \neq 0$, so $v(f_0) = r(v)(f_0)$. For $i > 0$, if there exists $s \in \RR^+$ such that $s \leq v(f_i)$,
then $r(v)(f_i) = v(f_i) \leq q_i r(v)(f_0) = q_i v (f_0)$; otherwise, for any $s \in \RR^+$
such that $s \leq r(v)(f_0)$, we have $v(f_i) < q_i s \leq q_i r(v)(f_0) = q_i v(f_0)$. It follows
that $v \in U$, so $r^{-1}(U \cap \Comm(A)) \subseteq U$.
\end{proof}

The following result is analogous to \cite[Proposition~2.6]{huber1}, although the proof is somewhat different.

\begin{lemma} \label{L:comm is spectral}
The spaces $\Sprb(A)$ and $\Comm(A)$ are spectral
and the map $r: \Sprb(A) \to \Comm(A)$ is spectral.
\end{lemma}
\begin{proof}
By Lemma~\ref{L:reified is spectral}, $\Sprv(A)$ is spectral.
Since $\Sprb(A)$ is closed in $\Sprv(A)$ for the patch topology, it is also spectral by Corollary~\ref{C:closed in spectral}. By the same reasoning, rational subspaces
of $\Sprb(A)$ are quasicompact.

By Lemma~\ref{L:strong rational subspaces} plus Lemma~\ref{L:retract}, $r: \Sprb(A) \to \Comm(A)$ is a continuous retraction and the inverse image of
any rational subspace is a rational subspace. Since rational subspaces of $\Sprb(A)$ are quasicompact, the same is true for $\Comm(A)$,  so $\Comm(A)$ is prespectral. It is also $T_0$, being a subspace of the $T_0$ space
$\Sprb(A)$.

Since rational subspaces of $\Comm(A)$ are quasicompact,
by Lemma~\ref{L:strong rational subspaces} the patch topology on $\Comm(A)$ is generated by rational subspaces and their complements.
By this remark plus Lemma~\ref{L:retract}, $r$ is continuous for the patch topologies, 
so $\Comm(A)$ is compact for the patch topology.
By Corollary~\ref{C:quasicompact is spectral}, $\Comm(A)$ is spectral.
Since $r$ is continuous for the patch topologies, Corollary~\ref{C:continuous is spectral}
implies that $r$ is spectral.
\end{proof}

\begin{remark}
In Huber's theory, including valuations which are not continuous would give rise to spaces which detect nontrivial blowups, i.e., analogues of the Riemann-Zariski spaces associated to schemes. For instance, for any two $f,g \in A$, some points at which both $f$ and $g$ vanish would be separated by assigning a nonzero but infinitesimal value to $f/g$. A similar effect would be achieved here by including reified valuations which are not commensurable; this point of view is taken in \cite{ducros-thuillier}.
\end{remark}

\section{Reified adic spectra}

We now define an intermediate construction between adic spectra and Gel'fand spectra, starting with an analogue of the definition of an affinoid f-adic ring.

\begin{defn}
By an \emph{affinoid seminormed ring}, we will mean a pair $(A^{\rhd}, A^{\Gr})$ in which $A^{\rhd}$
is a nonarchimedean seminormed ring and $A^{\Gr}$ is an integrally closed graded subring of $\Gr A^{\rhd}$. If $A^{\rhd}$ is separated and complete for its seminorm, we also call such a pair an \emph{affinoid Banach ring}.
For $r \in \RR^+$, write $A^{+,r}$ for the subring of $A^{\rhd,\circ,r}$ whose image in $\Gr^r A^{\rhd}$ belongs to $A^{\Gr,r}$.

A \emph{morphism} $(A^{\rhd}, A^{\Gr}) \to (B^{\rhd}, B^{\Gr})$ of affinoid seminormed rings is a bounded homomorphism $A^{\rhd} \to B^{\rhd}$ of nonarchimedean seminormed rings which induces a map $A^{\Gr} \to B^{\Gr}$.

For $(A^{\rhd}, A^{\Gr}) \to (B^{\rhd}, B^{\Gr})$, $(A^{\rhd}, A^{\Gr}) \to (C^{\rhd}, C^{\Gr})$ two morphisms of affinoid seminormed rings, define the \emph{tensor product}
$(B^{\rhd}, B^{\Gr}) \otimes_{(A^{\rhd}, A^{\Gr})} (C^{\rhd}, C^{\Gr})$ as the affinoid seminormed ring $(D^{\rhd}, D^{\Gr})$ with $D^{\rhd} = B^{\rhd} \otimes_{A^{\rhd}} C^{\rhd}$ (with the tensor product seminorm) and
$D^{\Gr}$ equal to the integral closure of the image of $B^{\Gr} \otimes_{A^{\Gr}} C^{\Gr}$ in $\Gr D^{\rhd}$.
\end{defn}

\begin{defn}
For $(A^{\rhd}, A^{\Gr})$ an affinoid seminormed ring, define the \emph{reified adic spectrum} (or for short the
\emph{readic spectrum}) of $(A^{\rhd}, A^{\Gr})$, denoted $\Spra(A^{\rhd}, A^{\Gr})$, as the subspace of
$\Comm(A^{\rhd})$ consisting of those reified valuations $v$ such that
for all $r \in \RR^+$ and $a \in A^{+,r}$, we have $v(a) \leq r$.
(Without this condition, we would only have $v(a) \leq r + \epsilon$ for all $\epsilon > 0$.)
A \emph{rational subspace} of $\Spra(A^{\rhd}, A^{\Gr})$ is the intersection 
with a rational subspace of $\Comm(A^{\rhd})$.
Restriction along a morphism $f: (A^{\rhd}, A^{\Gr}) \to (B^{\rhd}, B^{\Gr})$
of affinoid seminormed rings defines a continuous map
$f^*: \Spra(B^{\rhd}, B^{\Gr}) \to \Spra(A^{\rhd}, A^{\Gr})$.
\end{defn}

\begin{theorem} \label{T:readic spectrum is spectral}
For any affinoid seminormed ring $(A^{\rhd}, A^{\Gr})$, $\Spra(A^{\rhd}, A^{\Gr})$ is spectral
with a basis of quasicompact open subsets given by rational subspaces.
\end{theorem}
\begin{proof}
By Lemma~\ref{L:comm is spectral}, $\Comm(A^{\rhd})$ is spectral 
with a basis of quasicompact open subsets given by rational subspaces. Since
$\Spra(A^{\rhd}, A^{\Gr})$ is closed in $\Comm(A^{\rhd})$ for the patch topology, by
Corollary~\ref{C:closed in spectral} it is also spectral.
\end{proof}

Note that the map $\calM(A^{\rhd}) \to \Sprv(A^{\rhd})$ factors through $\Spra(A^{\rhd}, A^{\Gr})$. This has the following consequence.

\begin{lemma} \label{L:empty readic spectrum}
For any affinoid seminormed ring $(A^{\rhd}, A^{\Gr})$, 
the set $\Spra(A^{\rhd}, A^{\Gr})$ is empty if and only if $0$ is dense in $A^{\rhd}$.
(In particular, this condition does not depend on $A^{\Gr}$.)
\end{lemma}
\begin{proof}
Immediate from Theorem~\ref{T:Gel'fand spectrum}.
\end{proof}

We have the following analogue of Theorem~\ref{T:tensor product to fiber product}.
\begin{theorem} \label{T:tensor product to fiber product readic}
For $(A^{\rhd}, A^{\Gr}) \to (B^{\rhd}, B^{\Gr})$, $(A^{\rhd}, A^{\Gr}) \to (C^{\rhd}, C^{\Gr})$ morphisms of affinoid seminormed rings
and $(D^{\rhd}, D^{\Gr}) = (B^{\rhd}, B^{\Gr}) \otimes_{(A^{\rhd}, A^{\Gr})} (C^{\rhd}, C^{\Gr})$, the map
\[
\Spra(D^{\rhd}, D^{\Gr}) \to \Spra(B^{\rhd}, B^{\Gr}) \times_{\Spra(A^{\rhd}, A^{\Gr})}
\Spra(C^{\rhd}, C^{\Gr})
\]
is surjective.
\end{theorem}
\begin{proof}
Given $v_1 \in \Spra(B^{\rhd}, B^{\Gr})$, $v_2 \in \Spra(C^{\rhd}, C^{\Gr})$
mapping to $v \in \Spra(A^{\rhd}, A^{\Gr})$, 
Lemma~\ref{L:valuation product2} produces $v \in \Sprv(D^{\rhd})$ restricting to
$v_1 \in \Spra(B^{\rhd})$ and to $v_2 \in \Sprv(C^{\rhd})$. 
By the construction of $D^{\Gr}$, it is automatic that for any $r \in \RR^+$ and
any $a \in D^{+,r}$, we have $v(a) \leq r$. In particular, we have $v \in \Sprb(D^{\rhd})$.
By contrast, it is not automatic that $v$ is commensurable,
but we may enforce this by applying the map $r$ of 
Definition~\ref{D:retract}.
\end{proof}

Recall that Definition~\ref{D:real projection} gives rise to a projection map $\Spra(A^{\rhd}, A^{\Gr}) \to \calM(A^{\rhd})$. The fibers of this map may be described as follows.
\begin{defn}
For $F$ an ultrametric field, a \emph{graded valuation ring} of $F$ is a graded subring $R$ of 
$\Gr F$ with the property that for each $r>0$ and each nonzero $a \in \Gr^r F$, either $a$ or $a^{-1}$ (or both) belongs to $R$. In particular, the graded piece $R^1$ of $R$ is a valuation ring in the residue field $\Gr^1 F$ of $F$.
\end{defn}

\begin{lemma}  \label{L:graded valuation rings}
Let $(F^{\rhd}, F^{\Gr})$ be an affinoid seminormed ring such that $F^{\rhd}$ is an ultrametric field. Then there is a natural bijection between $\Spra(F^{\rhd}, F^{\Gr})$ and the set of graded valuation rings of $F^{\rhd}$ containing $F^{\Gr}$.
\end{lemma}
\begin{proof}
Given $v \in \Spra(F^{\rhd}, F^{\Gr})$, we may construct a graded valuation ring $R_v$ of $F^{\rhd}$ containing $F^{\Gr}$ as follows: for $a \in F$ with $\left| a \right| \leq r$, the class of $a$ in $\Gr^r F^{\rhd}$ belongs to $R_v^r$ if and only if $v(a) \leq r$. (To see that $R_v$ is indeed a graded valuation ring, note that if $a \in F$ satisfies $\left| a \right| \leq r$ but the class of $a$ does not belong to $R_v^r$,
then $\left| a \right| = r$ and $v(a) > r$, so $\left| a^{-1} \right| = r^{-1}$ and 
$v(a^{-1}) < r^{-1}$.)

In the other direction, given a graded valuation ring $R$ of $F^{\rhd}$ containing $F^{\Gr}$, 
for $a,b \in F^{\rhd}$ and $q > 0$, put
\[
v_{a,b,q} = \begin{cases} 1 & \mbox{if $b=0$;} \\
1 & \mbox{if $\left| b \right| > 0, \left| a \right| > q\left| b \right|$;} \\
1 & \mbox{if $\left| b \right| > 0, \left| a \right| = q\left| b \right|$, and $\overline{b/a} \in R^{1/q}$;} \\
0 & \mbox{otherwise}. 
\end{cases}
\]
By Lemma~\ref{L:reified order relation}, there exists a unique reified semivaluation $v$ such that $v_{a,b,q} = 1$ if and only if $v(a) \geq q v(b)$.
\end{proof}

\begin{cor} \label{C:graded valuation rings}
For $A = (A^{\rhd}, A^{\Gr})$ an affinoid seminormed ring and $\alpha \in \calM(A^{\rhd})$,
the construction of Lemma~\ref{L:graded valuation rings} defines a bijection between
the fiber of $\alpha$ under the projection $\Spra(A) \to \calM(A^{\rhd})$ of
Definition~\ref{D:real projection} and the set of graded valuation rings of $\calH(\alpha)$ containing the image of $A^{\Gr}$.
\end{cor}

\begin{defn}
For $A= (A^{\rhd}, A^{\Gr})$ an affinoid seminormed ring and $v \in \Spra(A)$
restricting to $\alpha \in \calM(A^{\rhd})$,
let $(\calH(v), \calH(v)^{\Gr})$ be the
affinoid seminormed ring in which $\calH(v) = \calH(\alpha)$ and $\calH(v)^{\Gr}$
is the graded valuation ring of $\calH(\alpha)$ corresponding to $v$
via Corollary~\ref{C:graded valuation rings}.
\end{defn}

\begin{remark} \label{R:projection of spectra}
Let $(A^{\rhd}, A^{\Gr})$ be an affinoid seminormed ring
such that $A^{\rhd}$ is an f-adic ring (this is not automatic; see Remark~\ref{R:norm on f-adic ring}),
and identify $\Gr^1 A^{\rhd}$ with $A^\circ/A^{\circ \circ}$.
We may then form an affinoid f-adic ring $(A^{\rhd}, A^+)$ by taking $A^+ = A^{+,1}$.
For this convention, there is a natural projection $\Spra(A^{\rhd}, A^{\Gr}) \to \Spa(A^{\rhd}, A^+)$, but unless $A^{\rhd}$ is Tate,
this map is not spectral or even continuous
because the inverse image of a rational subspace need not be a rational subspace
(see Remark~\ref{R:rational subspace definition}).
\end{remark}

\begin{remark} \label{R:projection of spectra2}
Conversely, let $(A^{\rhd}, A^+)$ be an affinoid f-adic ring for which $A^{\rhd}$ has been equipped with the structure of a nonarchimedean seminormed ring (this is always possible but not canonical; see Remark~\ref{R:norm on f-adic ring}). 
We may then form an affinoid seminormed ring by viewing
$A^+/A^{\circ \circ}$ as a graded subring of $\Gr A^{\rhd}$ concentrated in
$\Gr^1 A^{\rhd}$. Note that applying Remark~\ref{R:projection of spectra} then recovers $A^+$.
\end{remark}

\begin{remark} \label{R:canonical reification}
In Remark~\ref{R:projection of spectra2}, we may apply Lemma~\ref{L:lift reification}
to see that the projection map
$\Spra(A^{\rhd}, A^+) \to \Spa(A^{\rhd}, A^+)$ is surjective;
however, this map does not in general admits a distinguished section. 
One case where this does occur is when $A^{\rhd}$ is a Banach algebra over an ultrametric field $F$ with norm group
$\RR^+$:
the canonical reification of the norm on $F$ fixes a reification on every semivaluation.
For a related argument, see the proof of Lemma~\ref{L:valuation intersection}.
%
\end{remark}

\section{The structure presheaf on a readic spectrum}
\label{subsec:structure presheaf}

With Huber's adic spaces as a model, we now introduce the structure presheaf on a reified adic spectrum and build the reified analogues of adic spaces.
This development parallels the corresponding foundations in Huber's theory, for which we follow \cite[\S 2.4]{kedlaya-liu1} and sources cited therein.
However, as indicated in Remark~\ref{R:analytic adic},
the role of the Tate condition is largely eliminated by the presence of reifications.

Throughout \S\ref{subsec:structure presheaf}, let $(A^{\rhd}, A^{\Gr})$ be an affinoid Banach ring, and unless otherwise specified put $X = \Spra(A^{\rhd}, A^{\Gr})$.
We begin with the construction of homomorphisms corresponding to rational subspaces, as in 
Definition~\ref{D:rational subspace adic}
and Lemma~\ref{L:rational subspace}.
\begin{defn} \label{D:Gauss norm}
For $n$ a nonnegative integer, the \emph{standard Tate algebra} $A^{\rhd}\{T_1,\dots,T_n\}$ in the variables $T_1,\dots,T_n$ is the completion of $A^{\rhd}[T_1,\dots,T_n]$ for the \emph{Gauss norm} 
\[
\sum_{i_1,\dots,i_n=0}^\infty a_{i_1,\dots,i_n} T_1^{i_1} \cdots T_n^{i_n} \mapsto
\max\{\left| a_{i_1,\dots,i_n} \right|\};
\]
note that this definition is compatible with Definition~\ref{D:rational subspace adic}.
More generally, for $r_1,\dots,r_n > 0$, the \emph{weighted Tate algebra} $A^{\rhd}\{T_1/r_1,\dots,T_n/r_n\}$
in the variables $T_1,\dots,T_n$ is the completion of $A^{\rhd}[T_1,\dots,T_n]$ for the \emph{weighted Gauss norm} 
\[
\sum_{i_1,\dots,i_n=0}^\infty a_{i_1,\dots,i_n} T_1^{i_1} \cdots T_n^{i_n} \mapsto
\max\{\left| a_{i_1,\dots,i_n} \right| r_1^{i_1} \cdots r_n^{i_n}\}.
\]
\end{defn}

\begin{lemma} \label{L:valuation intersection}
For $x \in A^{\rhd}$, and $r>0$, $x \in A^{+,r}$ if and only if 
$v(x) \leq r$ for all $v \in \Spra(A^{\rhd}, A^{\Gr})$.
\end{lemma}
\begin{proof}
It is clear that if $x \in A^{+,r}$, then
$v(x) \leq r$ for all $v \in \Spra(A^{\rhd}, A^{\Gr})$.
Conversely, suppose that $x \notin A^{+,r}$. 
Let $B^{\rhd}$ be the completion of the group ring $A^{\rhd}[\RR^+]$ for the norm
taking $\sum_{r \in \RR^+} a_r [r]$ to $\max\{|a_r|r\}$.
Extend $B^{\rhd}$ to an affinoid f-adic ring $(B^{\rhd}, B^+)$
with $B^+ = \oplus_{s \in \RR^+} A^{+,s} [s^{-1}]$.
By construction, $x[r^{-1}] \notin B^+$, so
we may apply \cite[Lemma~3.3]{huber1} to produce $w \in \Spa(B^{\rhd}, B^{+})$
such that $w(x[r^{-1}]) > 1$. Let $v: A^{\rhd} \to \Gamma_{w,0}$ be the restriction of $w$, viewed as a reified valuation for the map $s \mapsto w([s])$; then 
$v \in \Spra(A^{\rhd}, A^{\Gr})$ and $v(x) > r$, as desired.
\end{proof}

\begin{defn} \label{D:rational subspace readic}
Consider a rational subspace $U$ of $X$ defined by parameters $f_0,\dots,f_n \in A^{\rhd}$ and scale factors $q_1,\dots,q_n > 0$ as in \eqref{eq:rational subspace reified}.
Let $B^{\rhd}$ be the quotient of $A^{\rhd}\{T_1/q_1,\dots,T_n/q_n\}$ by the closure of the ideal $(f_0 T_1 - f_1, \dots, f_0 T_n - f_n)$.
Let $B^{\Gr}$ be the integral closure of the image of $A^{\Gr}[T_1,\dots,T_n]$
in $\Gr B^{\rhd}$ (placing $T_i$ in degree $q_i$).
We now have a morphism $(A^{\rhd}, A^{\Gr}) \to (B^{\rhd}, B^{\Gr})$ of affinoid seminormed rings;
by Lemma~\ref{L:rational subspace readic} below, this construction depends only on the original rational subspace $U$ and not on the defining parameters.
\end{defn}

We have the following analogue of Lemma~\ref{L:rational subspace}.
\begin{lemma} \label{L:rational subspace readic}
Retain notation as in Definition~\ref{D:rational subspace readic}.
\begin{enumerate}
\item[(a)]
The morphism $(A^{\rhd}, A^{\Gr}) \to (B^{\rhd}, B^{\Gr})$ is initial among morphisms
$(A^{\rhd}, A^{\Gr}) \to (C^{\rhd}, C^{\Gr})$ for which $C^{\rhd}$ is complete and the image of $\Spra(C^{\rhd}, C^{\Gr})$ in 
$X$ is contained in $U$. (We will say for short that this morphism \emph{represents} $U$; we also characterize such a morphism as a \emph{rational localization}.)
\item[(b)]
The induced map $\Spra(B^{\rhd}, B^{\Gr}) \to U$ is a homeomorphism. More precisely, the rational subspaces of $\Spra(B^{\rhd}, B^{\Gr})$ correspond to the rational subspaces of $X$ contained in $U$.
\end{enumerate}
\end{lemma}
\begin{proof}
We follow the proofs of \cite[Proposition~1.3, Lemma~1.5]{huber2}.
Let $h: (A^{\rhd}, A^{\Gr}) \to (C^{\rhd}, C^{\Gr})$ be a morphism as in (a).
Then $v(h(f_0)) > 0$ for all $v \in \Spra(C^{\rhd}, C^{\Gr})$.
By Corollary~\ref{C:units}, we have $h(f_0) \in (C^{\rhd})^\times$.
For $i=1,\dots,n$, we have $v(h(f_i)/h(f_0)) \leq q_i$ for $v \in \Spra(C^{\rhd}, C^{\Gr})$.
By Lemma~\ref{L:valuation intersection},
we have $f_i/f_0 \in C^{+,q_i}$; we thus deduce (a).

To check (b), note the map $\Spra(B^{\rhd}, B^{\Gr}) \to U$ is injective
because any $v \in \Spra(B^{\rhd}, B^{\Gr})$ is uniquely determined by its
restriction to the image of $A^{\rhd}[T_1,\dots,T_n]$; it is surjective by (a)
applied to the map $(A^{\rhd}, A^{\Gr}) \to (\calH(v), \calH(v)^{\Gr})$
for each $v \in U$.
It is clear that every rational subspace of $X$ contained in $U$ pulls back to a
rational subspace of $\Spra(B^{\rhd}, B^{\Gr})$.
Conversely, given a rational subspace $V$ of $\Spra(B^{\rhd}, B^{\Gr})$ defined by some parameters $g_0,\dots,g_m \in B^{\rhd}$ and some scale factors $r_1,\dots,r_m > 0$, 
by Remark~\ref{R:rational subspace nearby generators} we may choose the $g_i$ to be in 
the image of $A^{\rhd}[T_1,\dots,T_n]$. By multiplying through by a suitable power of $f_0$, we obtain parameters in $A^{\rhd}$ itself, but these parameters need not generate the unit ideal in $A^{\rhd}$. However, since $\alpha(g_0) \neq 0$ for all $\alpha \in \calM(B^{\rhd})$, by compactness we can find $c>0$ such that $\alpha(g_0) \geq c$ for all $\alpha \in \calM(B^{\rhd})$. Put $g_{m+1} = 1$ and $r_{m+1} = c^{-1}$; then
the parameters $g_0,\dots,g_{m+1}$ and scale factors $r_1,\dots,r_{m+1}$ define a rational subspace of $X$ whose intersection with $U$ corresponds to $V$.
Since the intersection of two rational subspaces is again a rational subspace,
this completes the proof of (b).
\end{proof}

\begin{lemma} \label{L:covering faithful}
Let $\{(A^{\rhd}, A^{\Gr}) \to (B_i^{\rhd}, B_i^{\Gr})\}_i$ be the morphisms representing a finite cover $\frakU = \{U_i\}_i$ of $X$ by rational subspaces.
Then the image of $\Spec(\oplus_i B^{\rhd}_i) \to \Spec(A^{\rhd})$ contains $\Maxspec(A^{\rhd})$.
\end{lemma}
\begin{proof}
For each $\frakm \in \Maxspec(A)$, apply Corollary~\ref{C:units} to construct some $\alpha \in \calM(A)$ for which $\frakm = \ker(A \to \calH(\alpha))$. For some $i$, $\alpha$ extends to $\beta \in \calM(B_i)$, and $\ker(\beta)$ is a prime ideal of $B_i$ lifting $\frakm$.
\end{proof}

\begin{defn}
The \emph{structure presheaf} $\calO$ on $X$ assigns to each open subset $U$ the inverse limit of $B^{\rhd}$ over all homomorphisms $(A^{\rhd}, A^{\Gr}) \to (B^{\rhd}, B^{\Gr})$ representing rational subspaces of $X$ contained in $U$.
We say that $(A^{\rhd}, A^{\Gr})$ is \emph{sheafy} if $\calO$ is a sheaf; in this case,
$X$ is a locally ringed space by Lemma~\ref{L:henselian local ring} below.
\end{defn}

\begin{lemma} \label{L:henselian local ring}
For $v \in X$, the stalk $\calO_v$ is a henselian local ring whose residue field is dense in $\calH(v)$.
\end{lemma}
\begin{proof}
The local property follows from \cite[Corollary~2.3.7]{kedlaya-liu1}.
The henselian property follows from \cite[Lemma~2.2.3(a)]{kedlaya-liu1}.
\end{proof}

As in classical rigid geometry, most of our arguments about the structure presheaf involve a reduction to certain special types of coverings.
\begin{defn} \label{D:standard Laurent}
For $f_1,\dots,f_n \in A^{\rhd}$ generating the unit ideal and $q_1,\dots,q_n > 0$, the \emph{standard rational covering} of $X$ generated by $f_1,\dots,f_n$ with scale factors $q_1,\dots,q_n$ is the covering of $X$ by the rational subspaces
\[
U_i = \{v \in X: q_j v(f_j) \leq q_i v(f_i) \quad (j=1,\dots,n)\} \qquad (i=1,\dots,n).
\]
For $f_1,\dots,f_n \in A$ arbitrary and $q_1,\dots,q_n > 0$, the \emph{standard Laurent covering} generated by $f_1,\dots,f_n$ with scale factors $q_1,\dots,q_n$ is the covering by the rational subspaces
\[
S_e = \bigcap_{i=1}^n S_{i,e_i} \qquad (e = (e_1,\dots,e_n) \in \{-, +\}^n),
\]
where
\[
S_{i,-} =\{v \in X: v(f_i) \leq q_i\},
\quad
S_{i,+} =\{v \in X: v(f_i) \geq q_i\}.
\]
A standard Laurent covering with $n=1$ is also called a
\emph{simple Laurent covering}. 
\end{defn}

\begin{lemma} \label{L:rational and Laurent coverings}
The following statements hold.
\begin{enumerate}
\item[(a)]
Any finite covering of $X$ by rational subspaces can be refined by a standard rational covering. 
\item[(b)]
For any standard rational covering $\frakU$ of $X$, there exists a standard Laurent covering $\frakV$ of $X$ such that for each $V = \Spra(B^{\rhd},B^{\Gr}) \in \frakV$, the restriction of $\frakU$ to $V$ (omitting empty intersections) is a standard rational covering generated by units in $B^{\rhd}$.
\item[(c)]
Any standard rational covering of $X$ generated by units can be refined by a standard Laurent covering generated by units.
\end{enumerate}
\end{lemma}
\begin{proof}
To prove (a), we follow \cite[Lemma~8.2.2/2]{BGR}. 
We start with a finite covering
of $X$ by rational subspaces $U_1,\dots,U_n$, where $U_i$ is generated by the parameter set $S_i = \{f_{i0}, f_{i1},\dots,f_{in_i}\}$ with corresponding scale factors
$q_{i1}, \dots, q_{in_i}$. Let $S$ be the set of products of the form $s_1\cdots s_n$ where $s_i \in S_i$ for all $i$. Let $S'$ be the subset of $S$ consisting of products $s_1 \cdots s_n$ for which $s_i = f_{i0}$ for at least one $i$. 
Note that $S'$ generates the unit ideal: for any $v \in X$, for each $i$ we can find $s_i \in S_i$ not vanishing at $v$, taking $s_i = f_{i0}$ for any $i$ for which $v \in U_i$. Thus the parameter set $S'$ can be used to define a standard rational covering;
we do so by taking the scale factor associated to $f_{1j_1} \cdots f_{n j_n}$
to be $q_{1j_1} \cdots q_{nj_n}$.
To see that this refines the original covering, note that the rational subspace with first parameter $s_1\cdots s_n$ does not change if we add $S \setminus S'$ to the set of parameters (again because the $U_i$ form a covering), which makes it clear that this subspace is contained in $U_i$ for any index $i$ for which $s_i = f_{i0}$ (because we now have parameters obtained from $s_1,\dots,s_n$ by replacing $s_i$ with each of the other elements of $S_i$).

To prove (b), we follow \cite[Lemma~8.2.2/3]{BGR}.
Let $\frakU$ be the standard rational covering defined by the parameters $f_1,\dots,f_n$ with scale factors $q_1,\dots,q_n$.
Since $f_1,\dots,f_n$ generate the unit ideal, by Corollary~\ref{C:units}
the quantity
\[
c = \inf\{\max_i\{q_i \alpha(f_i)\}: \alpha \in \calM(A^{\rhd})\}
\]
is positive. In this case, the standard Laurent covering $\frakV$ defined by $f_1,\dots, f_n$ with scale factors $c/2, \dots, c/2$ has the desired property: on the subspace where $q_j \left| f_j \right| \leq c/2$ for $j=1,\dots,s$
and $q_i \left| f_i \right| \geq c/2$ for $i=s+1,\dots,n$, the restriction of $\frakU$ is the standard rational covering generated by $f_{s+1},\dots,f_n$ with scale factors $q_{s+1},\dots,q_n$ plus some empty intersections.

To prove (c), we follow \cite[Lemma~8.2.2/4]{BGR}. Consider the standard rational covering generated by the units $f_1,\dots,f_n$ with scale factors $q_1,\dots,q_n$. This cover is refined by the standard Laurent covering generated by $f_i f_j^{-1}$ with scale factors $q_i/q_j$ for $1 \leq i < j \leq n$, by an elementary combinatorics argument (any total ordering on a finite set has a maximal element). 
\end{proof}

This yields the following reduction argument,
analogous to \cite[Lemma~2.4.19]{kedlaya-liu1}.

\begin{lemma} \label{L:Tate reduction}
Let $\calP$ be a property of finite coverings of rational subspaces of $X$ by rational subspaces.
Suppose that $\calP$ satisfies the following condition.
\begin{enumerate}
\item[(a)]
The property $\calP$ is local: if it holds for a refinement of a given covering, it also holds for the original covering.
\item[(b)]
The property $\calP$ is transitive: if it holds for a covering $\{V_i\}_i$ of $U$ and for some coverings $\{W_{ij}\}_j$ of $V_i$, then it holds for the composite covering
$\{W_{ij}\}_{i,j}$ of $U$.
\item[(c)]
The property $\calP$ holds for any simple Laurent covering.
\end{enumerate}
Then the property $\calP$ holds for any finite covering of a rational subspace
of $X$ by rational subspaces.
\end{lemma}
\begin{proof}
This follows from Lemma~\ref{L:rational and Laurent coverings} as in 
\cite[Proposition~2.4.20]{kedlaya-liu1}.
\end{proof}

This yields the following criterion for sheafiness and acyclicity.
\begin{lemma} \label{L:acyclicity template}
Let $\calF$ be a presheaf of abelian groups on $X$ such that for every rational subspace $U$ of $X$ and every simple Laurent covering $V_1, V_2$ of $U$, we have 
\begin{equation} \label{eq:acyclicity template}
\check{H}^i(U, \calF; \{V_1,V_2\}) = \begin{cases} \calF(U) & \mbox{if $i=0$,} \\ 0 & \mbox{if $i=1$}.
\end{cases}
\end{equation}
Then for every rational subspace $U$ of $X$ and every finite covering $\frakV$ of $U$ by rational subspaces,
\[
H^i(U, \calF) = \check{H}^i(U, \calF; \frakV) = \begin{cases} \calF(U) & \mbox{if $i=0$,} \\ 0 & \mbox{if $i>0$}. \end{cases}
\]
\end{lemma}
\begin{proof}
This follows from Lemma~\ref{L:Tate reduction} as in
\cite[Proposition~2.4.21]{kedlaya-liu1}.
\end{proof}

Using this criterion, we may see that sheafiness implies acyclicity, by analogy with 
\cite[Theorem~2.4.23]{kedlaya-liu1}.
\begin{lemma} \label{L:H1 for simple Laurent covering}
Let $S_-, S_+$ be the simple Laurent covering of $X$ defined by some $f \in A^{\rhd}$ and some $q>0$. Let 
\[
(A^{\rhd}, A^{\Gr}) \to (B_1^{\rhd}, B_1^{\Gr}), (A^{\rhd},A^{\Gr}) \to (B_2^{\rhd}, B_2^{\rhd}), (A^{\rhd},A^{\Gr}) \to (B_{12}^{\rhd}, B_{12}^{\Gr})
\]
be the morphisms representing the rational subspaces $S_-, S_+, S_- \cap S_+$ of $X$.
Then the map $B_1^{\rhd} \oplus B_2^{\rhd} \to B_{12}^{\rhd}$ taking $(b_1, b_2)$ to $b_1 - b_2$ is surjective.
\end{lemma}
\begin{proof}
By Lemma~\ref{L:rational subspace readic}, we obtain strict surjections
\[
A^{\rhd}\{T/q\} \to B_1^{\rhd}, \quad A^{\rhd}\{U/q\} \to B_2^{\rhd}, \quad A^{\rhd}\{T/q,U/q^{-1}\} \to B_{12}^{\rhd}
\]
taking $T$ to $f$ and $U$ to $f^{-1}$. In particular, any $b \in B_{12}^{\rhd}$ can be lifted to some $\sum_{i,j=0}^\infty a_{ij} T^i U^j \in A^{\rhd}\{T/q,U/q^{-1}\}$. 
Let $a'_n$ be the sum of $a_{ij}$ over all $i,j \geq 0$ with $i-j=n$; note that this sum converges in $A^{\rhd}$. Let $b_1$ be the image of $\sum_{n=0}^\infty a'_n T^n$ in $B_1^{\rhd}$.
Let $b_2$ be the image of $-\sum_{n=1}^\infty a'_{-n} U^n$ in $B_2^{\rhd}$. Then
$(b_1, b_2) \in B_1^{\rhd} \oplus B_2^{\rhd}$ maps to $b \in B_{12}^{\rhd}$, proving the desired exactness.
\end{proof}

\begin{theorem} \label{T:Tate sheaf property for structure sheaf}
Suppose that $(A^{\rhd}, A^{\Gr})$ is sheafy. Then for every finite covering $\frakU$ of $X$ by rational subspaces,
\[
H^i(X, \calO) = \check{H}^i(X, \calO; \frakU) = 
\begin{cases} A^{\rhd} & \mbox{if $i=0$,} \\ 0 & \mbox{if $i>0$.} \end{cases}
\]
\end{theorem}
\begin{proof}
By Lemma~\ref{L:acyclicity template}, it suffices to check \v{C}ech-acyclicity for simple Laurent coverings. Since the sheafy condition propagates to rational subspaces, we may as well consider only simple Laurent coverings of $X$ itself. In the notation of Lemma~\ref{L:H1 for simple Laurent covering}, the sequence
\[
0 \to A^{\rhd} \to B_1^{\rhd} \oplus B_2^{\rhd} \to B_{12}^{\rhd} \to 0
\]
is exact at $B_{12}^{\rhd}$; by the sheafy hypothesis, it is also exact at $A^{\rhd}$ and $B_1^{\rhd} \oplus B_2^{\rhd}$.
\end{proof}

We may also establish a weak analogue of Kiehl's theorem on coherent sheaves,
by analogy with \cite[Theorem~2.7.7]{kedlaya-liu1}.
\begin{theorem} \label{T:weak Kiehl}
Suppose that $(A^{\rhd}, A^{\Gr})$ is sheafy. Then the global sections functor induces an equivalence of categories between $\calO$-modules which are locally free of finite rank and finite projective $A^{\rhd}$-modules.
\end{theorem}
\begin{proof}
We may use Lemma~\ref{L:Tate reduction} to reduce to the case where notation is as in Lemma~\ref{L:H1 for simple Laurent covering} and one is given a sheaf whose restrictions to $S_-, S_+, S_- \cap S_+$ correspond to finite projective modules $M_1, M_2, M_{12}$ over $B_1^{\rhd}, B_2^{\rhd}, B_{12}^{\rhd}$, respectively. In this case, 
by Theorem~\ref{T:Tate sheaf property for structure sheaf},
the diagram
\begin{equation} \label{eq:glueing square}
\xymatrix{
A^{\rhd} \ar[r] \ar[d]& B_1^{\rhd} \ar[d] \\
B_2^{\rhd} \ar[r] & B_{12}^{\rhd}
}
\end{equation}
forms a \emph{glueing square} in the sense of \cite[Definition~2.7.3]{kedlaya-liu1}.
We may thus appeal to
\cite[Proposition~2.7.5]{kedlaya-liu1} to deduce that the modules $M_1, M_2, M_{12}$ arise as base extensions of a finite projective module $M$ over $A^{\rhd}$.
\end{proof}

\begin{defn}
A \emph{locally reified valuation-ringed space}, or \emph{locally rv-ringed space} for short, is a locally ringed space $(X, \calO_X)$ equipped with the additional data of, for each $x \in X$, a reified valuation $v_x$ on the local ring $\calO_{X,x}$. A morphism of locally rv-ringed spaces $f: X \to Y$ is a morphism of locally ringed spaces
with the property that for each $x \in X$ mapping to $y \in Y$, the restriction of $v_x$ along the map $\calO_{Y,y} \to \calO_{X,x}$ is equal to $v_y$ as a reified valuation.
\end{defn}

\begin{defn}
Any locally rv-ringed space of the form $\Spra(A^{\rhd}, A^{\Gr})$ for some sheafy $(A^{\rhd}, A^{\Gr})$ is called an \emph{affinoid reified adic space}. 
For such a space, we recover $A^{\rhd}$ as the ring of global sections;
by Lemma~\ref{L:valuation intersection}, we may recover $A^{\Gr}$ from the reified valuations on local rings.

A locally v-ringed space which is covered by open subspaces which are affinoid reified adic spaces is called a \emph{reified adic space}. We suggest to abbreviate \emph{reified adic space} to \emph{readic space} or \emph{$\RR$-adic space}.
As in Remark~\ref{R:preadic}, one can formally define a ``space'' associated to a nonsheafy $(A^{\rhd}, A^{\Gr})$ using a functor of points approach, by analogy with  \cite[\S 8.2]{kedlaya-liu1}. 
\end{defn}

\section{On the sheafy condition}
\label{subsec:noetherian}

By analogy with Theorem~\ref{T:sheafy conditions},
we identify two important classes of sheafy affinoid seminormed rings. In the analogue of the strongly noetherian case, we also get a more precise analogue of Kiehl's characterization on coherent sheaves on affinoid spaces. We begin with the analogue of the stably uniform condition.

\begin{defn}
We say that a Banach ring $A$ is \emph{uniform} if its norm is equivalent to its spectral seminorm. (An equivalent condition is that there exists $c>0$ such that for all $a \in A^{\rhd}$, $\left| a^2 \right| \geq c \left| a \right|^2$.) We say that an affinoid Banach ring $(A^{\rhd}, A^{\Gr})$ is \emph{really stably uniform} if for any homomorphism $(A^{\rhd}, A^{\Gr}) \to (B^{\rhd}, B^{\Gr})$ representing a rational subspace of $X$, $B^{\rhd}$ is uniform. See \cite{buzzard-verberkmoes} for some exotic examples related to these conditions (e.g., for uniform rings which are not really stably uniform).
\end{defn}

\begin{lemma} \label{L:uniform ideal closure}
Let $A$ be a uniform Banach ring. For any $f \in A$, the ideals
\[
(T-f) \subseteq A^{\rhd}\{T/q\},
(1-fU) \subseteq A^{\rhd}\{U/q^{-1}\},
(T-f) \subseteq \frac{A^{\rhd}\{T/q,U/q^{-1}\}}{(TU-1)} 
\]
are closed.
\end{lemma}
\begin{proof}
As in \cite[Lemma~2.8.8]{kedlaya-liu1}.
\end{proof}

By analogy with Theorem~\ref{T:sheafy conditions}(b), we have the following.
\begin{theorem} \label{T:stably uniform sheafy}
If $(A^{\rhd}, A^{\Gr})$ is really stably uniform, then it is sheafy.
\end{theorem}
\begin{proof}
By Lemma~\ref{L:acyclicity template} and Lemma~\ref{L:H1 for simple Laurent covering},
it suffices to check that with notation as in Lemma~\ref{L:H1 for simple Laurent covering}, the sequence 
\[
0 \to A^{\rhd} \to B_1^{\rhd} \oplus B_2^{\rhd} \to B_{12}^{\rhd}
\]
is exact. We first check exactness at $A^{\rhd}$. By Theorem~\ref{T:Gel'fand spectrum},
for $a \in A^{\rhd}$,
\[
\left| a \right|_{A^{\rhd}, \mathrm{sp}} = \sup\{\alpha(a): \alpha \in \calM(A^{\rhd}) \} = \sup\{\alpha(a): \alpha \in \calM(B_1^{\rhd}) \cup \calM(B_2^{\rhd}) \}.
\]
In particular, if $a \in \ker(A^{\rhd} \to B_1^{\rhd} \oplus B_2^{\rhd})$, then
then $a$ has zero spectral seminorm; however, since $A$ is uniform by hypothesis,
this forces $a=0$.

We check exactness at $B_1^{\rhd} \oplus B_2^{\rhd}$
following \cite[\S 8.2.3]{BGR}. In the commutative diagram
\[
\xymatrix@C=15pt{
& & 0 \ar[d] & 0 \ar[d] & \\
& 0 \ar[r] \ar[d] & (T-f)A^{\rhd}\left\{\frac{T}{q}\right\} \oplus (1-fU)A^{\rhd}\left\{\frac{U}{q^{-1}} \right\} \ar[r] \ar[d] & (T-f) \frac{A^{\rhd}\left\{\frac{T}{q}, \frac{U}{q^{-1}}\right\}}{(TU-1)} \ar[d]\ar[r] & 0 \\
0 \ar[r] & A^{\rhd} \ar[r] \ar@{=}[d] & A^{\rhd}\left\{\frac{T}{q} \right\} \oplus A^{\rhd}\left\{\frac{U}{q^{-1}}\right\} \ar[r] \ar[d] & \frac{A^{\rhd}\left\{\frac{T}{q}, \frac{U}{q^{-1}}\right\}}{(TU-1)} \ar[r] \ar[d] & 0 \\
0 \ar[r] & A^{\rhd} \ar[r] \ar[d] & B_1^{\rhd} \oplus B_2^{\rhd} \ar[r] \ar[d] & B_{12}^{\rhd} \ar[r] \ar[d] & 0 \\
& 0 & 0 & 0 &
}
\]
the first two rows are clearly exact,
while the columns are exact by Lemma~\ref{L:uniform ideal closure}. By diagram chasing, we obtain exactness of the third row at $B_1^{\rhd} \oplus B_2^{\rhd}$.
\end{proof}

We next turn to the analogue of the strongly noetherian condition, where we can carry out a more thorough adaptation of Huber's constructions.
\begin{defn}
We say that a Banach ring $A$ is \emph{really strongly noetherian} if $A^{\rhd}\{T_1/r_1,\dots,T_n/r_n\}$ is noetherian for all $n \geq 0$ and all $r_1,\dots,r_n > 0$. We say that an affinoid Banach ring $(A^{\rhd}, A^{\Gr})$ is \emph{really strongly noetherian} if $A^{\rhd}$ is really strongly noetherian;
this implies that for every morphism
$(A^{\rhd}, A^{\Gr}) \to (B^{\rhd}, B^{\Gr})$ representing a rational subspace of $\Spra(A^{\rhd}, A^{\Gr})$, $B^{\rhd}$ is really strongly noetherian. 
\end{defn}

\begin{example} \label{exa:affinoid algebras}
Any ultrametric field is really strongly noetherian by 
\cite[Proposition~2.1.3]{berkovich1}, as then is any Berkovich affinoid algebra over an ultrametric field (see Definition~\ref{D:affinoid algebras}).
\end{example}

\begin{example}
Let $F$ be a ultrametric field with nontrivial norm which is perfect of characteristic $p$.
Let $W(F)$ be the ring of $p$-typical Witt vectors of $F$, which may be viewed as the unique $p$-adically separated and complete ring whose reduction modulo $p$ is $F$.
Each $x \in W(F)$ can be written uniquely as $\sum_{n=0}^\infty p^n [\overline{x}_n]$
with $\overline{x}_n \in F$ and brackets denoting Teichm\"uller lifts.
The set of $x \in W(F)$ for which $\lim_{n \to \infty} p^{-n} \left| \overline{x}_n \right| = 0$ is then a really strongly noetherian Banach ring for the norm $x \mapsto
\max_n \{p^{-n} \left| \overline{x}_n \right| \}$; see
\cite[Theorem~3.2]{kedlaya-noetherian}.
\end{example}

We mention one further class of examples.
\begin{theorem} \label{T:trivial to noetherian}
Let $A$ be a noetherian ring equipped with the trivial norm. Then $A$ is really strongly noetherian.
\end{theorem}
\begin{proof}
As in \cite[Theorem~3.2]{kedlaya-noetherian}, we use a Gr\"obner basis construction.
Choose $r_1, \dots, r_n > 0$.
Equip $\ZZ_{\geq 0}^n$ with the componentwise partial order $\leq$, and with
the graded lexicographic total ordering $\preceq$.
Since $\leq$ is a well-quasi-ordering (every sequence contains an infinite ascending subsequence) and $\preceq$ refines $\leq$, $\preceq$ is a well-ordering. 
For $x = \sum_I x_I T^I \in A\{T_1/r_1,\dots,T_n/r_n\}$ nonzero, 
consider those indices $I$ for which $x_I \neq 0$ and $r_1^{i_1} \cdots r_n^{i_n}$ is maximized, then identify the greatest such index with respect to $\preceq$;
we define the \emph{leading index} and \emph{leading coefficient}
of $x$ to be the resulting values of $I$ and $x_I$, respectively.

For $J$ an ideal of $A\{T_1/r_1,\dots,T_n/r_n\}$ and $I \in \ZZ_{\geq 0}^n$,
let $L_I$ be the ideal of $A$ consisting of 0 plus the leading coefficients of all elements of $J$
with leading index $I$. For $I_1 \leq I_2$ we have $L_{I_1} \subseteq L_{I_2}$.
Let $S$ be the set of indices $I$ for which $L_I \neq L_{I'}$ for any $I' < I$;
this set is finite by
the well-quasi-ordering property of $\leq$ and the noetherian property of $A$.
For each $I \in S$, let $G_I$ be a set of elements of $J$ realizing each 
leading coefficient in some finite set of generators of $L_I$. We may then present each element $x \in J$ as a linear combination of elements of $\cup_{I \in S} G_I$ by repeatedly applying the usual division algorithm as long as $x \neq 0$: identify the leading index of $x$ as a multiple of some element $I$ of $S$, then kill off the leading coefficient of $x$ by subtracting off a suitable monomial linear combination of elements of $G_{I}$.
\end{proof}
\begin{cor}
Let $A$ be the ring $\ZZ((z))$ equipped with the $z$-adic norm (for any normalization).
Then $A$ is really strongly noetherian.
\end{cor}
\begin{proof}
For $q = \left| z \right| \in (0,1)$, we have $A \cong 
\ZZ\{T/q, U/q^{-1}\}/(TU-1)$, which is really strongly noetherian by Theorem~\ref{T:trivial to noetherian}.
\end{proof}

One important consequence of the really strongly noetherian condition is that it allows topological considerations to be omitted from many algebraic constructions involving finitely generated modules.
\begin{lemma} \label{L:noetherian}
Let $A$ be a really strongly noetherian Banach ring.
\begin{enumerate}
\item[(a)]
Every ideal in $A$ is closed.
\item[(b)]
Every finite $A$-module is complete under the quotient topology induced by some (and hence any) surjection from a finite free module.
\item[(c)]
Every morphism between finite $A$-modules, topologized as in (b), is strict.
\end{enumerate}
\end{lemma}
\begin{proof}
Suppose first that $A$ is Tate. 
We first observe that if $M$ is a normed $A$-module whose completion $\widehat{M}$ is finitely generated, then $M = \widehat{M}$. This is proved as in \cite[Proposition~3.7.3/2]{BGR}: choose an $A$-linear surjection $f: A^n \to \widehat{M}$, apply
the Banach open mapping theorem for $A$ (see \cite{henkel}) to deduce that $f$ is strict,
then conclude by Nakamaya's lemma in the form of \cite[Lemma~1.2.4/6]{BGR}.

We now check that for any finite free $A$-module $F$, any submodule $M$ of $F$ is complete.
To wit, choose any $r \in (0,1)$ and put $B = A\{T/r,U/r^{-1}\}/(TU-1)$; then $B$ is necessarily Tate. By the previous paragraph, the image of $M \otimes_A B$ in $F \otimes_A B$ is closed. In particular, $M \otimes_A B$ and $\widehat{M} \otimes_A B$ have the same image in $F \otimes_A B$; since the map $A \to B$ of $A$-modules is split by the constant coefficient map, this implies that $M = \widehat{M}$.

By taking $F = A$ in the previous paragraph, we deduce (a). To check (b), let $A^n \to M$ be a surjection of $A$-modules and apply the previous paragraph to $\ker(A^n \to M)$.
To check (c), we may again extend scalars from $A$ to $B$ and apply the Banach open mapping theorem.
\end{proof}

As a consequence, we obtain some results on flatness of certain ring homomorphisms.
\begin{cor} \label{C:series ring flat}
Let $A$ be a really strongly noetherian Banach ring. Then for all $n \geq 0$ and
$r_1,\dots,r_n > 0$, the morphism 
\[
A[T_1,\dots,T_n] \to A\{T_1/r_1,\dots,T_n/r_n\}
\]
of rings is flat.
\end{cor}
\begin{proof}
By induction, we reduce to the case $n=1$ and put $T = T_1, r = r_1$.
To handle this case, we follow \cite[Lemma~1.7.6]{huber-book}.

We first prove that $A \to A\{T/r\}$ is flat.
For $M$ a finite $A$-module, by Lemma~\ref{L:noetherian}, $M$ is complete for its natural topology and any finite presentation of $M$ is strict.
We may thus identify
\[
M \otimes_A A\{T/r\} \cong M \widehat{\otimes}_A A\{T/r\} \cong M\{T/r\},
\]
where $M\{T/r\}$ denotes the set of formal sums $\sum_{i=0}^\infty m_i T^i$ with $m_i \in M$ such that for some (hence any) norm on $M$ induced by a presentation, 
$\lim_{i \to \infty} \left| m_i \right| r^{i} = 0$.
For any short exact sequence $0 \to M \to N \to P \to 0$, it is clear that
\[
0 \to M\{T/r\} \to N\{T/r\} \to P\{T/r\} \to 0
\]
is exact; this proves that $A \to A\{T/r\}$ is flat.

Since $A \to A[T]$ and $A \to A\{T/r\}$ are both flat, 
by standard commutative algebra
(see for instance \cite[0.10.2.5]{ega3-1}),
the verification that $A[T] \to A\{T/r\}$ is flat reduces
to showing that for each $\frakm \in \Maxspec(A)$, for $k = A/\frakm$, the morphism
$A[T] \otimes_A k \to A\{T/r\} \otimes_A k$ is flat.
By the previous paragraph, the target of this map may be identified with $k\{T/r\}$, which as a module over the principal ideal domain $k[T]$ is torsion-free (since it embeds into $k \llbracket T \rrbracket$) and hence flat.
\end{proof}

\begin{cor} \label{C:noetherian flat1}
Let $(A^{\rhd}, A^{\Gr})$ be a really strongly noetherian affinoid Banach ring.
\begin{enumerate}
\item[(a)]
For any rational localization $(A^{\rhd}, A^{\Gr}) \to (B^{\rhd}, B^{\Gr})$,
the map $A^{\rhd} \to B^{\rhd}$ is flat.
\item[(b)]
Let $\{(A^{\rhd}, A^{\Gr}) \to (B_i^{\rhd}, B_i^{\Gr})\}_i$ be the morphisms representing a finite cover $\frakU = \{U_i\}_i$ of $X$ by rational subspaces.
Then the morphism $A^{\rhd} \to \oplus_i B_i^{\rhd}$ is faithfully flat.
\end{enumerate}
\end{cor}
\begin{proof}
To prove (a),
choose a presentation of $B^{\rhd}$ as in 
Definition~\ref{D:rational subspace readic}.
By Corollary~\ref{C:series ring flat},
the map 
\[
A^{\rhd}[T_1,\dots,T_n] \to A^{\rhd}\{T_1/q_1,\dots,T_n/q_n\}
\]
is flat, as then is the map
\[
A^{\rhd}[f_0^{-1}] \cong \frac{A^{\rhd}[T_1,\dots,T_n]}{(f_0 T_1 - f_1, \dots, f_0 T_n - f_n)}
\to 
\frac{A^{\rhd}\{T_1/q_1,\dots,T_n/q_n\}}{(f_0 T_1 - f_1, \dots, f_0 T_n - f_n)}
\cong B_i^{\rhd}
\]
(applying Lemma~\ref{L:noetherian} to obtain the last isomorphism).
Since the ordinary localization $A^{\rhd} \to A^{\rhd}[f_0^{-1}]$ is flat, so then is $A^{\rhd} \to B_i^{\rhd}$. This proves (a), from which (b) follows by invoking Lemma~\ref{L:covering faithful} and some standard commutative algebra
(see for instance \cite[Tag 00HQ]{stacks-project}).
\end{proof}

\begin{cor} \label{C:noetherian injective}
Let $(A^{\rhd}, A^{\Gr})$ be a really strongly noetherian affinoid Banach ring.
Then for every rational subspace $U$ of $X$ and every finite covering $\frakV$ of $U$ by rational subspaces, the maps
\[
\calF(U) \to \check{H}^0(U, \calO; \frakV), \quad
\calF(U) \to H^0(U, \calO)
\]
are injective. (They will be shown to be bijective in Theorem~\ref{T:very strongly noetherian}.)
\end{cor}
\begin{proof}
Immediate from Corollary~\ref{C:noetherian flat1}.
\end{proof}

We mention also a refinement of Corollary~\ref{C:noetherian flat1}, which gives a stronger result but has a somewhat mysterious extra hypothesis.

\begin{lemma} \label{L:noetherian flat1}
Let $(A^{\rhd}, A^{\Gr})$ be a really strongly noetherian affinoid Banach ring.
Let $(A^{\rhd}, A^{\Gr}) \to (B^{\rhd}, B^{\Gr})$ be a rational localization.
Suppose that $\frakm \in \Maxspec(A^{\rhd})$ has the property
that $A^{\rhd}/\frakm \cong \calH(\beta)$ for some $\beta \in \calM(B^{\rhd})$.
Then for every positive integer $n$,
the map $A^{\rhd}/\frakm^n \to B^{\rhd}/\frakm^n B^{\rhd}$ is an isomorphism.
\end{lemma}
\begin{proof}
We follow \cite[Proposition~7.2.2/1]{BGR}.
By Lemma~\ref{L:noetherian}, the ideal $\frakm^n$ is closed; we may thus form 
the commutative diagram
\[
\xymatrix{
A^{\rhd} \ar[r] \ar[d] & B^{\rhd} \ar[d] \ar@{-->}[ld] \\
A^{\rhd}/\frakm^n \ar[r] & B^{\rhd}/\frakm^n B^{\rhd}
}
\]
of Banach rings. 
The dashed arrow exists and is unique for $n=1$ by hypothesis, and hence for all $n$ by
 the universal property of rational localizations.
Consequently, surjectivity of $B^{\rhd} \to B^{\rhd}/\frakm^n B^{\rhd}$ implies surjectivity of $A^{\rhd}/\frakm^n \to B^{\rhd}/\frakm^n B^{\rhd}$. On the other hand,
surjectivity of $A^{\rhd} \to A^{\rhd}/\frakm^n$ implies the surjectivity of $B^{\rhd} \to A^{\rhd}/\frakm^n$;
since $\ker(B^{\rhd} \to B^{\rhd}/\frakm^n B^{\rhd}) = \frakm^n B^{\rhd}$ is contained in the kernel of $B^{\rhd} \to A^{\rhd}/\frakm^n$ (it being generated by elements of said kernel), it follows that $A^{\rhd}/\frakm^n \to B^{\rhd}/\frakm^n B^{\rhd}$ is also injective.
\end{proof}

\begin{remark} \label{R:noetherian flat1}
For a given pair $(A^{\rhd}, A^{\Gr})$,
one can deduce Corollary~\ref{C:noetherian flat1}(a) 
from Lemma~\ref{L:noetherian flat1} if for every maximal ideal $\frakm$ of $A^{\rhd}$, the Banach ring $A^{\rhd}/\frakm$ is (isomorphic to) an ultrametric field; that is, its norm is equivalent to a multiplicative norm. 
This holds for classical affinoid algebras
(see Definition~\ref{D:affinoid algebras}),
but in light of Remark~\ref{R:ultrametric field} it is unclear to what extent it should occur more generally.
\end{remark}

We end up with the following analogue of Theorem~\ref{T:sheafy conditions}(a).
\begin{theorem} \label{T:very strongly noetherian}
Let $(A^{\rhd}, A^{\Gr})$ be a really strongly noetherian affinoid Banach ring.
Then $(A^{\rhd}, A^{\Gr})$ is sheafy.
\end{theorem}
\begin{proof}
By Definition~\ref{D:rational subspace readic} and Lemma~\ref{L:rational subspace readic}, the really strongly noetherian property propagates to rational subspaces.
We may thus follow the proof of Theorem~\ref{T:stably uniform sheafy}
after replacing
Theorem~\ref{T:Gel'fand spectrum} with Corollary~\ref{C:noetherian injective}
and Lemma~\ref{L:uniform ideal closure} with
Lemma~\ref{L:noetherian}.
\end{proof}

We may also upgrade Theorem~\ref{T:weak Kiehl} to obtain an extension of Kiehl's theorem on coherent sheaves (see Remark~\ref{R:Kiehl} and Remark~\ref{R:no Kiehl}).
\begin{theorem} \label{T:stably noetherian Kiehl}
Let $(A^{\rhd}, A^{\Gr})$ be a really strongly noetherian affinoid Banach ring.
Then the global sections functor induces an equivalence of categories between coherent $\calO$-modules
and finite $A^{\rhd}$-modules.
\end{theorem}
\begin{proof}
Again, we may use Lemma~\ref{L:Tate reduction} to reduce to the case where notation is as in Lemma~\ref{L:H1 for simple Laurent covering} and  and one is given a sheaf $\calF$ whose restrictions to $S_-, S_+, S_- \cap S_+$ correspond to finite modules $M_1, M_2, M_{12}$ over $B_1^{\rhd}, B_2^{\rhd}, B_{12}^{\rhd}$, respectively.
By Theorem~\ref{T:very strongly noetherian}, $(A^{\rhd}, A^{\Gr})$ is sheafy, so
the diagram \eqref{eq:glueing square} is again a glueing square in the sense of \cite[Definition~2.7.3]{kedlaya-liu1}.
Let $M$ be the kernel of the map $M_1 \oplus M_2 \to M_{12}$ given by $(m_1, m_2) \mapsto m_1 - m_2$. By \cite[Lemma~2.7.4]{kedlaya-liu1},  the sequence 
\[
0 \to M \to M_1 \oplus M_2 \to M_{12} \to 0
\]
is exact and the induced maps $M \otimes_{A^{\rhd}} B_*^{\rhd} \to M_*$ for $* = 1,2,12$ are surjective; however, it does not immediately follow that these maps are injective or that $M$ is finitely generated.

However, we do know that the sheaf $\calF$ is globally finitely generated (by some finitely generated submodule of $M$),
so we may choose a surjection $\calO^{\oplus n} \to \calF$ of $\calO$-modules.
Let $\calG$ be the kernel of this surjection and put $N_1 = \calG(S_-), N_2 = \calG(S_+), N_{12} = \calG(S_- \cap S_+)$. By definition, the sequences
\[
0 \to N_* \to B_*^{\rhd \oplus n} \to M_* \to 0\qquad (* = 1,2,12)
\]
are exact; by the really strongly noetherian hypothesis, $N_*$ is a finitely generated $B_*^{\rhd}$-module. By Corollary~\ref{C:noetherian flat1},
the induced maps $N_i \otimes_{B_i^{\rhd}} B_{12}^{\rhd} \to N_{12}$ are isomorphisms.
We may thus repeat the previous argument to see that $\calG$ is globally finitely generated; that is, there exists an exact sequence of the form
\[
\calO^{\oplus m} \to \calO^{\oplus n} \to \calF \to 0.
\]
We may now take global sections to obtain a finite $A^{\rhd}$-module 
$\coker(A^{\rhd \oplus m} \to A^{\rhd \oplus n})$ whose associated sheaf is isomorphic (by the right exactness of tensor products) to $\calF$.
\end{proof}

\begin{remark} \label{R:Kiehl}
In the case of a Tate affinoid algebra over an ultrametric field (see 
Definition~\ref{D:has enough seminorms2}), Theorem~\ref{T:stably noetherian Kiehl} specializes to Kiehl's original glueing theorem for coherent sheaves 
\cite[Theorem~9.4.3/3]{BGR}, modulo the comparison of Grothendieck topologies (Theorem~\ref{T:affinoid patch dense1}). The case of a Berkovich affinoid algebra reduces to the case of a Tate affinoid algebra using the technique of Lemma~\ref{L:valuation intersection}, again modulo comparison of topologies (Theorem~\ref{T:affinoid patch dense2}).
\end{remark}

\begin{remark} \label{R:no Kiehl}
The analogue of Theorem~\ref{T:stably noetherian Kiehl} for affinoid f-adic rings 
would state that for $(A^{\rhd}, A^+)$ an affinoid f-adic ring such that $A^{\rhd}$ is strongly noetherian, the global sections functor induces an equivalence of categories between coherent $\calO$-modules on
$\Spa(A^{\rhd}, A^+)$ and finite $A^{\rhd}$-modules.
The status of this statement is unclear to us; the special case where $A^{\rhd}$ is Tate will be treated in an upcoming sequel to \cite{kedlaya-liu1}, while the special case where $A^{\rhd}$ is really strongly noetherian follows from
Theorem~\ref{T:stably noetherian Kiehl}. Some additional results in this direction can be found in \cite{fujiwara-kato}.
\end{remark}

\section{Comparison of Grothendieck topologies}
\label{subsec:comparison}

We now study the relationship between Gel'fand spectra and readic spectra, following the study of the relationship between rigid and adic spaces made by Huber \cite[\S 4]{huber1} and van der Put and Schneider \cite{vanderput-schneider}.

\begin{defn} 
For $A$ a nonarchimedean normed ring, 
define the \emph{strictly special G-topology} (resp. the \emph{special G-topology}) on
$\calM(A)$ by taking the admissible open subsets to be the finite unions of strictly rational subspaces
(resp.\ rational subspaces) of $\calM(A)$ and taking the admissible coverings to be the finite set-theoretic coverings. Both G-topologies are prespectral;
the special G-topology is also $T_0$. 
\end{defn}

\begin{defn} \label{D:has enough seminorms1}
For $A = (A^{\rhd}, A^{\Gr})$ an affinoid normed ring,
let $i: \calM(A^{\rhd}) \to \Spra(A)$ be the natural inclusion obtained by viewing each real seminorm as a reified semivaluation. This map is continuous for the special G-topology  on $\calM(A^{\rhd})$, but not the natural topology.
\end{defn}

\begin{defn} \label{D:has enough seminorms2}
For $A = (A^{\rhd}, A^{+})$ an affinoid f-adic ring,
view $A^{\rhd}$ as a nonarchimedean normed ring via Remark~\ref{R:norm on f-adic ring}.
Let $j: \calM(A^{\rhd}) \to \Spa(A)$ be the natural map obtained by viewing each real seminorm as a semivaluation. 
If $A^{\rhd}$ is Tate (but not necessarily otherwise; see Remark~\ref{R:rational subspace definition}), this map is continuous for the strictly special G-topology  on $\calM(A^{\rhd})$, but not the natural topology.
\end{defn}

In general, the maps $i,j$ do not hit enough points of $\Spra(A)$ or $\Spa(A)$ to make it possible to recover the structure of these spaces from $\calM(A^{\rhd})$. One crucial exception is the case of classical affinoid algebras.
For the remainder of \S\ref{subsec:comparison}, let $F$ be an ultrametric field.
\begin{defn} \label{D:affinoid algebras}
A \emph{Tate affinoid algebra} over $F$ is a Banach algebra $A$ over $F$ which can be realized as a topological quotient of $F\{T_1,\dots,T_n\}$ for some $n$.
If the norm on $F$ is nontrivial, then every maximal ideal of $A$ has residue field finite over $F$ \cite[Corollary~6.1.2/3]{BGR}, so we obtain a natural inclusion $\Maxspec(A) \to \calM(A)$.

A \emph{Berkovich affinoid algebra} over $F$ is a Banach algebra over $F$ which can be realized as a topological quotient of $F\{T_1/r_1,\dots,T_n/r_n\}$ for some $n$ and some $r_1,\dots,r_n > 0$.
\end{defn}

\begin{theorem} \label{T:affinoid patch dense1}
Assume that the norm on $F$ is nontrivial, and let $A$ be a reduced Tate affinoid algebra over $F$.
\begin{enumerate}
\item[(a)]
The Banach ring $A$ is uniform.
\item[(b)]
For any homomorphism $(A, A^\circ) \to (B, B^+)$ representing a rational subspace of $\Spa(A,A^\circ)$, we have $B^+ = B^\circ$.
\item[(c)]
The image of the composition $\Maxspec(A) \to \calM(A) \stackrel{j}{\to} \Spa(A,A^\circ)$
is dense for the patch topology.
\item[(d)]
Equip $\calM(A)$ with the strictly special G-topology. Equip $\Maxspec(A)$ with the subspace topology from $\calM(A)$. Then the images of $\Maxspec(A)$ and $\calM(A)$ under the functor $\catPrespec \to \catSpec$ of Corollary~\ref{C:prespec to spec} may be naturally identified with each other and with $\Spa(A, A^{\circ})$.
\end{enumerate}
\end{theorem}
\begin{proof}
For (a), see \cite[Theorem 6.2.4/1]{BGR}.
For (b), see \cite[Lemma~2.5.9]{kedlaya-liu1}.
For (c), see \cite[Corollary~2.5.13]{kedlaya-liu1} or \cite[Corollary~4.2]{huber1}.
For (d), one may either see \cite[Corollary~4.4, Corollary~4.5]{huber1} or argue as follows.
Note that by construction, the map $\Maxspec(A) \to \Spa(A,A^{\circ})$ is spectral and $\Maxspec(A)$ admits a basis of quasicompact open subsets
each of which is the inverse image of a quasicompact open subset of $\Spa(A, A^{\circ})$.
By (c) and Remark~\ref{R:adjunction}, we obtain a natural isomorphism $\Spec(D(\Maxspec(A))) \cong \Spa(A, A^{\circ})$;
by similar reasoning, we obtain the isomorphism
$\Spec(D(\calM(A))) \cong \Spa(A, A^{\circ})$.
\end{proof}


\begin{theorem} \label{T:affinoid patch dense2}
Let $A$ be a reduced Berkovich affinoid algebra over $F$.
\begin{enumerate}
\item[(a)]
The Banach ring $A$ is uniform.
\item[(b)]
For any homomorphism $(A, \Gr A) \to (B, B^{\Gr})$ representing a rational subspace of $\Spra(A,\Gr A)$, we have $B^{\Gr} = \Gr B$.
\item[(c)]
The images of $i: \calM(A) \to \Spra(A,\Gr A)$
and $j: \calM(A) \to \Spa(A, A^\circ)$ are dense for the patch topologies.
\item[(d)]
Equip $\calM(A)$ with the strictly special G-topology. 
Then the image of  $\calM(A)$ under the functor $\catPrespec \to \catSpec$ of Corollary~\ref{C:prespec to spec} may be naturally identified with $\Spa(A, A^\circ)$.
\item[(e)]
Equip $\calM(A)$ with the special G-topology. 
Then the image of  $\calM(A)$ under the functor $\catPrespec \to \catSpec$ of Corollary~\ref{C:prespec to spec} may be naturally identified with $\Spra(A, \Gr A)$.
\end{enumerate}
\end{theorem}
\begin{proof}
For (a), see \cite[Proposition~2.1.4(ii)]{berkovich1}.
To prove (b), we first observe that if $A$ is a topological quotient of $F\{T_1/r_1,\dots,T_n/r_n\}$ for some $r_1,\dots,r_n > 0$ in the divisible closure of $\left| F^\times \right|$, then $A$ is a Tate affinoid algebra.
We next observe that if $r>0$ is not in the divisible closure of $\left| F^\times \right|$, then $E = F\{T/r, U/r^{-1}\}/(TU-1)$ is again an ultrametric field
and $\Gr E = (\Gr F)[\overline{T}^{\pm}]$ with $\overline{T}$ placed in degree $r$. Now consider a homomorphism as in (b),
and put $A_E = A \widehat{\otimes}_F E$ and $B_E = B \widehat{\otimes}_F E$.
Then on one hand, $\Gr A_E = (\Gr A) \otimes_{\Gr F} \Gr E$. On the other hand, if we put
$B_E^{\Gr} = B^{\Gr} \otimes_{\Gr F} \Gr E$, then $B_E^{\Gr}$ is integrally closed,
$(A_E, \Gr A_E) \to (B_E, B_E^{\Gr})$ again represents a rational subspace of
$\Spra(A_E, \Gr A_E)$ (described by the same parameters), and
$B_E^{\Gr} = \Gr B_E$ if and only if $B^{\Gr} = \Gr B$.
We may thus reduce (b) to Theorem~\ref{T:affinoid patch dense1}(b).

To prove (c), we treat only the case of $i$, the case of $j$ being similar (and easier).
It suffices to check that for $V \subseteq U$ an inclusion of rational subspaces of $\Spra(A, \Gr A)$ with $i^{-1}(U) = i^{-1}(V)$, we must have $U = V$.
Let $(A, \Gr A) \to (B, B^{\Gr}) \to (C, C^{\Gr})$ be the representing homomorphisms.
By (b), $B^{\Gr} = \Gr B$ and $C^{\Gr} = \Gr C$, so it suffices to check that $B=C$.
However, for any ultrametric field $E$ containing $F$, by \cite[Lemma~2.2.9]{kedlaya-liu1} we can check that $B \to C$ is an isomorphism by checking that $B \widehat{\otimes}_F E \to C \widehat{\otimes}_F E$ is an isomorphism. As in (b), we may thus reduce to the case where $U, V$ are strictly rational subspaces; in this case we may appeal directly to \cite[Corollary~2.5.13]{kedlaya-liu1} to conclude.

The proofs of (d) and (e) are similar, so we omit the former. To prove (e),
note that $i$ is spectral and $\calM(A)$ admits a basis of quasicompact open subsets for the special G-topology
each of which is the inverse image of a quasicompact open subspace of $\Spra(A, \Gr A)$.
By (c), $i$ also has dense image under the patch topology on $\Spra(A, \Gr A)$. By Remark~\ref{R:adjunction}, we obtain a natural isomorphism $\Spec(D(\calM(A))) \cong \Spra(A, \Gr A)$,
as desired.
\end{proof}

\begin{remark} \label{R:embedding}
The map $j$ is injective when $F$ has nontrivial norm (see Remark~\ref{R:canonical reification}), but may not be injective when the norm on $F$ is trivial (see Example~\ref{exa:disc2}).
\end{remark}

\begin{remark}
The conclusion of Theorem~\ref{T:affinoid patch dense1}(b) holds also for affinoid subdomains; see \cite[Proposition~2.5.14(a)]{kedlaya-liu1}. One may similarly extend
Theorem~\ref{T:affinoid patch dense2}(b) to affinoid subdomains; we omit futher details.
\end{remark}

\begin{remark} \label{R:ACVF}
As in Remark~\ref{R:extend two valuations}, the arguments found in \cite{huber1}
in the direction of Theorem~\ref{T:affinoid patch dense1} rely on elimination of quantifiers in the first-order 
theory of algebraically closed valued fields (ACVF); see especially the proof of
\cite[Theorem~4.1]{huber1}. 
One can take a similar approach to Theorem~\ref{T:affinoid patch dense2} by establishing elimination of quantifiers in the theory of algebraically closed reified valued fields (which we propose to call ACRVF); this should follow easily from the corresponding result for ACVF since one is simply adding one constant to the language corresponding to the image of each positive real number in the value group.
(A distinct but possibly related theory is the theory ACV$^2$F of \cite[\S 8]{hrushovski-loeser}.) On the other hand, it is also possible to deduce 
Theorem~\ref{T:affinoid patch dense2} directly from elimination of quantifers in ACVF, by making a base extension from $F$ to a suitably large overfield as in the proof of Lemma~\ref{L:valuation intersection}; this approach is the one taken in
\cite{ducros-thuillier}.
\end{remark}

\section{Closed unit discs}

We illustrate the previous discussion by making all of the constructions explicit in
a simple but instructive case. The reader may find it useful to contrast this situation with the corresponding picture in the case of adic spectra
\cite[Example~2.20]{scholze1}.

\begin{example} \label{exa:disc1}
Let $K$ be an algebraically closed ultrametric field with residue field $k$,
and equip $A = K[T]$ with the Gauss norm.
The structure of $\calM(A)$ is well-known; it is a contractible space which is an inverse limit
of finite trees. 
A detailed treatment can be found in \cite[Chapter~1]{baker-rumely};
here, we only mention that each point of $\calM(A)$ is of exactly one of the following types.
(This labeling is due to Berkovich \cite[Proposition~1.4.4]{berkovich1}.)
\begin{enumerate}
\item[1.]
A semivaluation factoring through $A/\frakm$ for some maximal ideal $\frakm$.
\item[2.]
The $\rho$-Gauss valuation on $K[T-z]$ for some $z \in K$ and some $\rho \in (0,1] \cap |K^\times|$. (This includes the Gauss valuation, for which $\rho=1$; the corresponding point is called the \emph{Gauss point} or \emph{maximal point}.)
\item[3.]
The $\rho$-Gauss valuation on $K[T-z]$ for some $z \in K$ and some $\rho \in (0,1] \setminus |K^\times|$.
\item[4.]
None of the above. Any such valuation can be interpreted as the infimum of the supremum valuations
over a decreasing sequence of closed discs with empty intersection.
\end{enumerate}
For any choice of $A^+$, the natural map $\calM(A) \to \Spa(A,A^+)$ is injective
(see Remark~\ref{R:embedding}).
The points not in the image form a fifth type.
\begin{enumerate}
\item[5.]
A valuation of rank 2 which specializes a point of type 2.
\end{enumerate}
To describe these points more explicitly, choose $x \in \calM(A)$ of type 2 for some particular $z, \rho$. The residue field $k_x$ of $\calH(x)$ can then be identified with
$k(\overline{T})$ with $\overline{T}$ being the class of $(T-z)/\lambda$ for some 
$\lambda \in K$ with $\left| \lambda \right| = \rho$. This defines an identification of $k_x$ with the function field of $\mathbb{P}^1_k$, but this identification can be modified by changing the choices of $z, \lambda$; consequently, only the point at infinity on $\mathbb{P}^1_k$ is distinguished. What we can do canonically is to identify the finite places of $k_x$ with the branches of $\calM(A)$ below $x$; each such place then defines a discrete valuation on $k_x$, which we may compose with $x$ to form a reified valuation of rank 2
specializing $x$. 
(For example, the branch of $\calM(A)$ at $x$ containing the type 1 point defined by the ideal $(T-z)$
corresponds to a specialization of $x$ in which the valuation of $(T-z)/\lambda$ changes from being equal to 1 to being infinitesimally smaller than 1.)
If $x$ is not the Gauss point, then the infinite place of $k_x$ corresponds to the branch of $\calM(A)$ above $x$, and we similarly obtain one more type 5 point specializing $x$.
By contrast, if $x$ is the Gauss point, one gets additional points of type 5 specializing $x$
if and only if $A^+ \neq A^\circ$; see Example~\ref{exa:upper branch} for a typical example.

For any choice of $A^{\Gr}$ whose $r=1$ component equals $A^+/A^{\circ\circ}$,
there is also a natural map $\Spa(A,A^+) \to \Spra(A,A^{\Gr})$. The complement of the image of this map
consists of points of a sixth type.
\begin{enumerate}
\item[6.]
A valuation of rank 2 which specializes a point of type 3.
\end{enumerate}
For $x$ of type 3, there are exactly two points of type 6 specializing to $x$,
corresponding to the branches of $\calM(A)$ above and below $x$.
To wit, if $x$ is defined by some $z,\rho$, then $T-z$ has valuation equal to $\rho$ according to $x$, but infinitesimally larger or smaller than $\rho$ according to the specializations.
\end{example}


\begin{remark} \label{R:no density in disc example}
The difference between $\Spra(A,A^{\Gr})$ and the rational subspace
$\{v \in \Spra(A,A^{\Gr}): v(T) \leq 1\}$ consists only of those points of type 5 specializing the Gauss point not corresponding to branches below $x$.
In particular, $\calM(A)$ does not meet this difference. 
\end{remark}

\begin{example} \label{exa:upper branch}
In Example~\ref{exa:disc1}, the ring $\Gr A$ may be identified with the polynomial ring
$(\Gr K)[T]$ with $T$ placed in degree $r=1$. Under this identification, take $A^{\Gr} = \Gr K \subset \Gr A$. Let $\Gamma$ be the reified value group of the valuation on $K$, and let $\Gamma'$ be the lexicographic product $\Gamma \times \RR$ equipped with the reification inherited from $\Gamma$. Then $\Spra(A, A^{\Gr})$ contains a unique reified valuation $v$ with values in $\Gamma'$ extending the valuation on $K$ and sending $T$ to $(0, 1)$; this is a type 5 point of $\Spra(A, A^{\Gr})$ not corresponding to a branch below the Gauss point and not satisfying $v(T) \leq 1$.
\end{example}

\begin{example}  \label{exa:disc2}
Keep notation from Example~\ref{exa:disc1}, but now with the trivial norm on $K$.
In this case, the structure of $\calM(A)$ is simpler: the tree consists
of branches corresponding to elements of $K$, meeting at the Gauss point.
The lower endpoints of each branch is of type 1; the Gauss point is of type 2; other points are of type 3.
There is again an embedding $\calM(A) \to \Spra(A,A^{\Gr})$; 
the complement of its image consists of points of types 5 (specializing the Gauss point) and 6 (two for each type 3 point).
However, one cannot fit $\Spa(A,A^{\circ})$ in between; it arises
from $\Spra(A,A^{\Gr})$ by removing the type 6 points, then collapsing
the interior of each branch to a point.
\end{example}

\section{Perfectoid algebras and their spectra}
\label{sec:perfectoid}

To conclude, we quickly redevelop the theory of perfectoid algebras in the context of reified adic spectra, following \cite{kedlaya-liu1}. (See \cite{scholze1} and \cite{gabber-ramero-arxiv} for other treatments.) Throughout \S\ref{sec:perfectoid}, fix a prime number $p$.

\begin{defn}
By a \emph{perfect uniform affinoid Banach algebra over $\FF_p$}, we will mean an affinoid seminormed ring $(R^{\rhd}, R^{\Gr})$ such that $R^{\rhd}$ is a perfect (i.e., the Frobenius map is bijective) uniform Banach algebra over $\FF_p$ (viewed as an ultrametric field using the trivial norm). Note that this forces $R^{\Gr}$ to also be perfect.
Note also that $R^{\rhd}$ cannot be both perfect and noetherian unless it is a finite direct sum of perfect fields.
\end{defn}

\begin{theorem} \label{T:perfect uniform}
Let $(R^{\rhd}, R^{\Gr})$ be a perfect uniform affinoid Banach algebra over $\FF_p$.
Let $f: (R^{\rhd}, R^{\Gr}) \to (S^{\rhd}, S^{\Gr})$ be a morphism of affinoid seminormed rings satisfying one of the following conditions.
\begin{enumerate}
\item[(a)]
The morphism $f$ represents a rational subspace of $\Spra(S^{\rhd}, S^{\Gr})$.
\item[(b)]
The homomorphism $R^{\rhd} \to S^{\rhd}$ is finite \'etale
and $S^{\Gr}$ is the integral closure of $R^{\Gr}$ in $\Gr S^{\rhd}$.  
\end{enumerate}
Then $(S^{\rhd}, S^{\Gr})$ is also a perfect uniform affinoid Banach algebra over $\FF_p$.
In particular, $(R^{\rhd}, R^{\Gr})$ is really stably uniform, hence sheafy by
Theorem~\ref{T:stably uniform sheafy}.
\end{theorem}
\begin{proof}
Part (a) is proved as in \cite[Proposition~3.1.7]{kedlaya-liu1}.
Part (b) is a consequence of \cite[Theorem~3.1.15]{kedlaya-liu1}.
\end{proof}

\begin{defn}
A uniform affinoid Banach algebra $(A^{\rhd}, A^{\Gr})$ over $\QQ_p$ is \emph{perfectoid} if 
for all $r \in \RR^+$ and $x \in A^{+, r}$,
there exists $y \in A^{\rhd, \circ, r^{1/p}}$ such that
$x-y^p \in A^{+,r/p}$. Note that this forces $y \in A^{+,r^{1/p}}$
because $A^{\Gr}$ is integrally closed.
\end{defn}

\begin{lemma} \label{L:perfectoid reduction}
A uniform affinoid Banach algebra $(A^{\rhd}, A^{\Gr})$ over $\QQ_p$ is perfectoid if 
and only if it satisfies the following conditions.
\begin{enumerate}
\item[(a)] The Frobenius map on $A^{\rhd,\circ}/(p)$ is surjective.
\item[(b)] There exists $x \in A^{\rhd,\circ}$ with $x^p - p \in p^2 A^{\rhd,\circ}$.
\end{enumerate}
In particular, the perfectoid condition depends only on $A^{\rhd}$ and is consistent with the definition in \cite{kedlaya-liu1}.
\end{lemma}
\begin{proof}
As in \cite[Proposition~3.6.2(d)]{kedlaya-liu1}.
\end{proof}

\begin{defn} \label{D:perfectoid to char p}
Given a uniform affinoid Banach algebra $(A^{\rhd}, A^{\Gr})$ over $\QQ_p$, we may construct a perfect uniform affinoid Banach algebra $(R^{\rhd}, R^{\Gr})$ over $\FF_p$ as follows.
Define the underlying multiplicative monoid $R^{\rhd}$ to be the inverse limit of $A^{\rhd}$ under the $p$-power map. We define the addition on $R^{\rhd}$ by the formula
\[
(x_n)_n + (y_n)_n = \left( \lim_{m \to \infty} (x_{m+n} + y_{m+n})^{p^m} \right)_n.
\]
One checks easily that this gives $R^{\rhd}$ the structure of a perfect uniform Banach algebra over $\FF_p$ with respect to the norm $\left| (x_n)_n \right| = \left| x_n \right|$. Similarly, define the underlying multiplicative monoid $R^{\Gr}$ to be the inverse limit of $A^{\Gr}$ under the $p$-power map; again, one checks easily that this gives
$(R^{\rhd}, R^{\Gr})$ the structure of a perfect uniform affinoid Banach algebra over $\FF_p$.
\end{defn}

\begin{defn} \label{D:perfectoid to char 0}
Let $(R^{\rhd}, R^{\Gr})$ be a perfect uniform affinoid Banach algebra over $\FF_p$.
An element $z = \sum_{n=0}^\infty p^n [\overline{z}_n] \in W(R^{+})$ is \emph{primitive of degree $1$} if the following conditions hold:
\[
\overline{z}_0 \in R^{\rhd, \times} \cap R^{+,1/p},  \quad
\overline{z}_0^{-1} \in R^{+,p},  \quad
\overline{z}_1 \in (R^{+})^\times.
\]
In this case, we can form a uniform affinoid Banach algebra $(A^{\rhd}, A^{\Gr})$ over $\QQ_p$
by setting $A^{\rhd} = W(R^+)[[\overline{z}]^{-1}]/(z)$ and taking $A^{\Gr,r}$ to be the image of $R^{+,r}$ under the composition of the Teichm\"uller map $R^{\rhd} \to W(R^+)[[\overline{z}]^{-1}]$ with the projection to $\Gr A^{\rhd}$. 

Conversely, with notation as in Definition~\ref{D:perfectoid to char p},
the map $\theta: W(R^+) \to A^+$ induced by the multiplicative map $R^+ \to A^+$ taking
$(x_n)_n$ to $x_0$ is surjective and its kernel is principal with a generator which is primitive of degree 1 \cite[Lemma~3.6.3]{kedlaya-liu1}.
\end{defn}

\begin{theorem} \label{T:perfectoid correspondence}
The constructions of Definitions~\ref{D:perfectoid to char p} and~\ref{D:perfectoid to char 0} define quasi-inverse functors which give equivalences of categories between the category of perfectoid uniform affinoid Banach algebras $A$ over $\QQ_p$ and pairs
$(R, I)$ where $R = (R^{\rhd}, R^{\Gr})$ is a perfect uniform affinoid Banach algebra over $\FF_p$ and $I$ is an ideal of $W(R^+)$ generated by an element which is primitive of degree $1$.
\end{theorem}
\begin{proof}
As in \cite[Theorem~3.6.5]{kedlaya-liu1}.
\end{proof}

\begin{defn}
Suppose that $A$ and $(R,I)$ correspond as in Theorem~\ref{T:perfectoid correspondence}.
Then $A$ is an ultrametric field if and only if $R$ is; consequently, in general we obtain a natural bijection $\Spra(A) \to \Spra(R)$.
\end{defn}

\begin{theorem} \label{T:perfectoid correspondence2}
Suppose that $A$ and $(R,I)$ correspond as in Theorem~\ref{T:perfectoid correspondence}.
\begin{enumerate}
\item[(a)]
The map $\Spra(A) \to \Spra(R)$ is a homeomorphism. 
\item[(b)]
For $U \subseteq \Spra(A)$ and $V \subseteq \Spra(R)$ which correspond, $U$ is a rational subspace if and only if $V$ is.
\item[(c)]
With notation as in (b), let $A \to B$ and $R \to S$ be the morphisms representing $U$ and $V$. Then $B$ is a perfectoid uniform affinoid Banach algebra over $\QQ_p$, $S$ is a perfect uniform affinoid Banach algebra over $\FF_p$, and $B$ and $(S, I \cdot W(S^+))$ correspond as in Theorem~\ref{T:perfectoid correspondence}.
\end{enumerate}
\end{theorem}
\begin{proof}
As in \cite[Theorem~3.6.14]{kedlaya-liu1}.
\end{proof}

\begin{theorem} \label{T:perfectoid correspondence3}
Suppose that $A$ and $(R,I)$ correspond as in Theorem~\ref{T:perfectoid correspondence}.
\begin{enumerate}
\item[(a)]
Let $B^{\rhd}$ be a finite \'etale $A^{\rhd}$-algebra viewed as a uniform Banach algebra
(see \cite[Proposition~2.8.16]{kedlaya-liu1}) and let $B^{\Gr}$ be the integral closure of $A^{\Gr}$ in $\Gr A^{\rhd}$. Then $B = (B^{\rhd}, B^{\Gr})$ is again a perfectoid uniform affinoid Banach algebra over $\QQ_p$.
\item[(b)]
The functors of Theorem~\ref{T:perfectoid correspondence} induce equivalences of categories of objects $B$ as in (a) and pairs $(S, I W(S^+))$ where $R \to S$ is a morphism as in Theorem~\ref{T:perfect uniform}(b) (so $S^{\rhd}$ is a finite \'etale $R^{\rhd}$-algebra).
\end{enumerate}
\end{theorem}
\begin{proof}
This follows from \cite[Theorem~3.6.21]{kedlaya-liu1} and Lemma~\ref{L:perfectoid reduction}, without any further arguments required.
\end{proof}

\begin{theorem}
Let $A \to B$, $A \to C$ be morphisms of perfectoid uniform affinoid Banach algebras.
Let $(R, I)$ be the pair corresponding to $A$ via Theorem~\ref{T:perfectoid correspondence}, then apply the correspondence to $A \to B$, $A \to C$ to obtain morphisms $R \to S, R \to T$. Then $B \widehat{\otimes}_A C$ is again perfectoid and the map
$A \to B \widehat{\otimes}_A C$ corresponds via Theorem~\ref{T:perfectoid correspondence} to the map $R \to S \widehat{\otimes}_R T$; moreover, the tensor product norm on $B \otimes_A C$ induced by the spectral norms on $B$ and $C$ coincides with the spectral norm.
\end{theorem}
\begin{proof}
This is the analogue of \cite[Proposition~3.6.11]{kedlaya-liu1}, but the proof of that statement is incomplete, so a corrected argument is needed.
Note first that as in \cite[Example~3.6.6]{kedlaya-liu1}, both claims hold in the case
\[
B = A\{T_1/\rho_1,\dots,T_n/\rho_n\}, \quad C = A\{T'_1/\rho'_1,\dots,T'_{n'}/\rho'_{n'}\}
\]
with
\[
S = R\{T_1/\rho_1,\dots,T_n/\rho_n\}, \quad T = R\{T'_1/\rho'_1,\dots,T'_{n'}/\rho'_{n'}\}.
\]
We may thus reduce the general case to the case where $A \to B$, $A \to C$ factor through surjections $B' \to B, C' \to C$. Using \cite[Remark~3.1.6, Lemma~3.3.9]{kedlaya-liu1}, we see that each of these surjections is \emph{almost optimal}: the quotient norm induced by the spectral norm on the source coincides with the spectral norm on the target. We thus deduce both claims in the general case.
\end{proof}

\begin{defn}
A \emph{reified perfectoid space} is a reified adic space over $\QQ_p$ which is covered by the readic spectra of perfectoid affinoid Banach algebras over $\QQ_p$.
Using Theorem~\ref{T:perfectoid correspondence}, Theorem~\ref{T:perfectoid correspondence2}, and Theorem~\ref{T:perfectoid correspondence3}, we may construct a ``tilting'' correspondence between reified perfectoid spaces and perfect uniform readic spaces over $\FF_p$, which induces homeomorphisms of underlying topological spaces and of \'etale topoi.
\end{defn}

\end{document}